\documentclass[12pt]{amsart}
\usepackage[utf8]{inputenc}
\usepackage[english]{babel}
\usepackage[usenames,dvipsnames]{color}
\usepackage{mathtools, amsmath, amssymb, enumitem, bbm, mathrsfs, amsthm, bm}
\usepackage{graphicx}
\usepackage{cite}
\usepackage{tikz-cd}
\usepackage{verbatim}
\usepackage{xcolor}
\usetikzlibrary{arrows,arrows.meta}
\usepackage[hidelinks]{hyperref}
\usepackage{wrapfig}
\usepackage[normalem]{ulem}
\usepackage[left=1in, right=1in, bottom=1.5in, top=1.5in]{geometry}
\usepackage{setspace}

\tikzcdset{arrow style=tikz, diagrams={>=stealth'}}
\numberwithin{equation}{section}
\counterwithin{figure}{section}

\newcommand{\overbar}[1]{\mkern 1.5mu\overline{\mkern-1.5mu#1\mkern-1.5mu}\mkern 1.5mu}

\newcommand{\Z}{\mathbb{Z}}

\newcommand{\concat}{\op{.}}
\newcommand{\op}{\operatorname}

\newcommand{\exampleqed}{\nobreak \ifhmode \par \fi \hfill $\blacksquare$}
\newcommand\stdpt{
\begin{tikzpicture}
    \node[inner sep=0pt] (bullet) {\scalebox{1.3}{$\bullet$}};
\end{tikzpicture}
}
\newcommand\critpt{
\begin{tikzpicture}
    \node[inner sep=0pt] (bullet) {\scalebox{1.3}{$\bullet$}};
    \node[draw,circle,inner sep=1pt, minimum size=12pt, line width=0.4mm, yshift=0.35mm] at (bullet) {};
  \end{tikzpicture}
}

\newtheorem*{repeatmaintheorem}{Theorem \ref{thm:main theorem}}

\newtheorem*{repeatskeletoncorollary}{Corollary \ref{cor:skeleton theorem}}

\newtheorem*{repeatpropdualtrees}{Proposition \ref{prop:dual trees connected by simple moves}}

\newtheorem*{repeatdeftropicalcorrespondence}{Definition \ref{def:tropical correspondence}}

\newtheorem{theorem}{Theorem}[section]

\theoremstyle{plain}
\newtheorem{proposition}[theorem]
{Proposition}

\theoremstyle{plain}
\newtheorem{lemma}[theorem]
{Lemma}

\theoremstyle{plain}
\newtheorem{corollary}[theorem]
{Corollary}

\theoremstyle{definition}
\newtheorem{definition}[theorem]
{Definition}

\theoremstyle{definition}
\newtheorem{example}[theorem]
{Example}

\theoremstyle{definition}
\newtheorem{remark}[theorem]
{Remark}

\begin{document}

\title{Boundary stratifications of Hurwitz spaces}
\author{Darragh Glynn}
\address{Mathematics Institute\\University of Warwick}

\begin{abstract}
\begin{sloppypar} 
Let $\mathcal{H}$ be a Hurwitz space that parametrises holomorphic maps to $\mathbb{P}^1$. Abramovich, Corti and Vistoli, building on work of Harris and Mumford, describe a compactification $\overline{\mathcal{H}}$ with a natural boundary stratification. We show that the irreducible strata of $\overline{\mathcal{H}}$ are in bijection with combinatorial objects called decorated trees (up to a suitable equivalence), and that containment of irreducible strata is given by edge contraction of decorated trees. This combinatorial description allows us to define a tropical Hurwitz space, which we identify with the skeleton of the Berkovich analytification of $\overline{\mathcal{H}}$. The tropical Hurwitz space that we obtain is a refinement of a version defined by Cavalieri, Markwig and Ranganathan. We also provide an implementation that computes the stratification of $\overline{\mathcal{H}}$, and discuss applications to complex dynamics. 
\end{sloppypar}
\end{abstract}

\maketitle

\begin{section}{Introduction}
A \textit{Hurwitz space} $\mathcal{H}$ is a moduli space parametrising holomorphic maps of compact Riemann surfaces (with certain points of the source and target surfaces marked) that satisfy prescribed branching conditions, up to pre- and post-composition with isomorphisms. Hurwitz spaces are complex orbifolds that trace back to Hurwitz \cite{hurwitz1891} and their geometry is governed by rich combinatorics \cite{fulton69}, \cite{romagny2006hurwitz}, \cite{mochizuki1995geometry}.

Harris and Mumford give a compactification of $\mathcal{H}$ by a space of \textit{admissible covers} -- which are maps between nodal Riemann surfaces -- in order to study the moduli space $\mathcal{M}_{g}$ of genus-$g$ Riemann surfaces \cite{harris1982kodaira}. While the space of admissible covers is not normal, its normalisation $\overline{\mathcal{H}}$ is smooth, and Abramovich, Corti and Vistoli give a moduli interpretation to $\overline{\mathcal{H}}$ as a space that parametrises \textit{balanced twisted covers} \cite{abramovich2003twisted}. The boundary $\overline{\mathcal{H}} \smallsetminus \mathcal{H}$ is a normal crossings divisor, and this endows $\overline{\mathcal{H}}$ with a stratification by locally closed strata.

The compactification $\overline{\mathcal{H}}$ is an analogue of the Deligne-Mumford compactification $\overline{\mathcal{M}}_{g,n}$ of the moduli space of genus-$g$ Riemann surfaces with $n$ marked points. In both cases, when marked points collide, they `bubble off' onto a new component Riemann surface. In $\overline{\mathcal{M}}_{g,n}$ this produces a nodal Riemann surface, and in $\overline{\mathcal{H}}$ this produces a certain type of map between nodal Riemann surfaces.

The boundary $\overline{\mathcal{M}}_{g,n} \smallsetminus \mathcal{M}_{g,n}$ is also a normal crossings divisor, endowing $\overline{\mathcal{M}}_{g,n}$ with a stratification. In fact, if $\mathcal{H}$ parametrises maps whose target is a genus-$g$, $n$-marked Riemann surface, then the natural forgetful map to $\mathcal{M}_{g,n}$ is a finite covering map that extends to the boundary $\pi_\text{tgt}: \overline{\mathcal{H}} \rightarrow \overline{\mathcal{M}}_{g,n}$ and respects the stratifications of both spaces: a stratum of $\overline{\mathcal{H}}$ is a connected component of the preimage of a stratum of $\overline{\mathcal{M}}_{g,n}$ under $\pi_\text{tgt}$.

The stratification of $\overline{\mathcal{M}}_{g,n}$ is well understood: the strata are in bijection with isomorphism classes of \textit{$n$-marked genus-$g$ stable graphs} (see Section \ref{subsubsec:stable graphs}), and containment of closures of strata corresponds to edge contraction of these graphs.

In this paper, for $\mathcal{H}$ a Hurwitz space that parametrises maps to $\mathbb{P}^1$, we provide an analogous description of the strata of $\overline{\mathcal{H}}$ and of containment of closures of strata. The prescribed branching conditions that determine $\mathcal{H}$ are encoded in a \textit{portrait} (Definition \ref{def:portrait}), and for $P$ a portrait, we define combinatorial objects called \textit{$P$-decorated trees} (Definition \ref{def:decoration}) and an equivalence relation on $P$-decorated trees called \textit{Hurwitz equivalence} (Definition \ref{def:hurwitz equivalence of decorated trees}). We assign to each point of $\overline{\mathcal{H}}$ a Hurwitz equivalence class of $P$-decorated trees, and we prove that two points of $\overline{\mathcal{H}}$ are in the same stratum if and only if they are assigned the same Hurwitz equivalence class of $P$-decorated trees. The theorem is proved in Section \ref{sec:decorated trees from multicurves}.

\begin{theorem}\label{thm:main theorem}
Let $\mathcal{H}$ be a Hurwitz space that parametrises maps to $\mathbb{P}^1$, with $P$ its defining portrait. The irreducible strata of the compactification $\overline{\mathcal{H}}$ are in bijection with the set $\op{Stab}_{0,B}(P)$ of $P$-decorated trees up to Hurwitz equivalence. Furthermore, containment of closures of strata in $\overline{\mathcal{H}}$ corresponds to edge contraction of $P$-decorated trees.
\end{theorem}

A $P$-decorated tree consists of an abstract tree together with a cyclic ordering of the half-edges around each vertex and an assignment of a permutation to each half-edge (with some cycles of the permutations possibly labelled). This object keeps track of the monodromy information associated to a holomorphic map of Riemann surfaces as marked points collide.

It is straightforward to convert the theorem to an algorithm that computes the stratification of $\overline{\mathcal{H}}$ -- code is available on request, and will soon be available on the author's website.

\begin{remark}
Our $P$-decorated trees are a refinement of the \textit{combinatorial admissible covers} defined by Cavalieri, Markwig and Ranganathan \cite{cavalieri2016tropicalizing}. In their paper, they assign to each point of $\overline{\mathcal{H}}$ a combinatorial admissible cover. Points that share a given combinatorial admissible cover form a union of strata of the same dimension -- but not, in general, an irreducible stratum. That is, points may have the same combinatorial admissible cover but live in different strata, or even in different connected components of $\overline{\mathcal{H}}$ (see Example \ref{ex:introduction example}). We relate $P$-decorated trees to combinatorial admissible covers in Section \ref{subsec:relation to combinatorial admissible covers}.
\end{remark}

\begin{example}\label{ex:introduction example}
\begin{sloppypar}
We illustrate Theorem \ref{thm:main theorem} now with an example. Let ${F: \mathbb{P}^1 \rightarrow \mathbb{P}^1}$ be a rational map of degree 4 that has three critical points of local degree 2 and one critical point of local degree 4, mapping to distinct points $b_1, b_2, b_3, b_4$, respectively. The preimages of $b_1$, $b_2$, and $b_3$ also contain two non-critical points each. Let $a_1, a_2, a_3$ be non-critical preimages of $b_1, b_2, b_3$, respectively. This is a description in words of the following portrait $P$:
\[
\begin{tikzcd}[column sep=small, row sep=12mm, column sep=2.7mm, ampersand replacement=\&]
\critpt\arrow[dr,"\scriptstyle 2" left, line width=0.5mm] 
\&
\overset{\displaystyle a_1}{\stdpt}\arrow[d,"\scriptstyle 1" right, line width=0.5mm]
\&
\stdpt\arrow[dl,"\scriptstyle 1" right, line width=0.5mm]
\& \& ~\&
\critpt\arrow[dr, "\scriptstyle 2" left, line width=0.5mm] 
\& 
\overset{\displaystyle a_2}{\stdpt}\arrow[d,"\scriptstyle 1" right, line width=0.5mm]
\&
\stdpt\arrow[dl,"\scriptstyle 1" right, line width=0.5mm]
\& \& ~\&
\critpt\arrow[dr,"\scriptstyle 2" left, line width=0.5mm]
\& 
\overset{\displaystyle a_3}{\stdpt}\arrow[d,"\scriptstyle 1" right, line width=0.5mm]
\&
\stdpt\arrow[dl,"\scriptstyle 1" right, line width=0.5mm]
\& \& ~\&
\& 
\critpt\arrow[d,"\scriptstyle 4" left, line width=0.5mm]
\&
\\
\&
\underset{\displaystyle b_1}{\stdpt} 
\& \& \& \& \&
\underset{\displaystyle b_2}{\stdpt} 
\& \& \& \& \&
\underset{\displaystyle b_3}{\stdpt} 
\& \& \& \& \&
\underset{\displaystyle b_4}{\stdpt} 
\&
\end{tikzcd}
\]

There is a Hurwitz space, which we denote by $\mathcal{H}$, that parametrises the collection of tuples ${(F: \mathbb{P}^1 \rightarrow \mathbb{P}^1, b_1, b_2, b_3, b_4, a_1, a_2, a_3)}$ fitting the portrait $P$, up to pre- and post-composition with M\"obius transformations. $\mathcal{H}$ has two connected components; each is isomorphic to a Riemann sphere with 18 punctures. Compactifying fills in these punctures, and so $\overline{\mathcal{H}}$ has 38 irreducible strata: the 2 open connected components of $\mathcal{H}$ and the 36 punctures.

Two $P$-decorated trees that represent distinct punctures of $\mathcal{H}$ are shown in Figure \ref{fig:introduction Hurwitz space example d=4}. These punctures have the same combinatorial admissible cover in the sense of \cite{cavalieri2016tropicalizing} (see Section \ref{subsec:relation to combinatorial admissible covers}), but in fact belong to different connected components of $\overline{\mathcal{H}}$. 
\end{sloppypar}

\begin{figure}[ht!]
\centering
\includegraphics[scale=1]{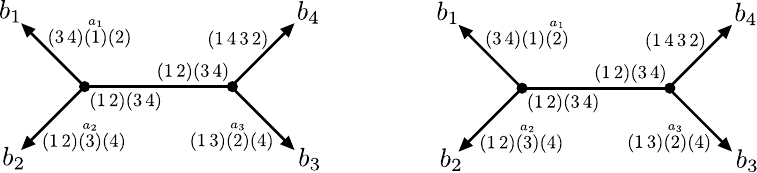}
\caption{Two decorated trees that represent distinct strata of $\overline{\mathcal{H}}$ in Example \ref{ex:introduction example}. The only difference between the decorations is the position of the label $a_1$.}
\label{fig:introduction Hurwitz space example d=4}
\end{figure}

\begin{sloppypar}
Roughly speaking, a puncture parametrises a `degenerate' map obtained by pinching a simple loop $\Delta$ on $\mathbb{P}^1 \smallsetminus \{b_1, b_2, b_3, b_4\}$ and its preimage $F^{-1}(\Delta)$ on ${\mathbb{P}^1 \smallsetminus \{a_1, a_2, a_3\}}$. To construct a $P$-decorated tree that represents the puncture, we take an embedded dual tree of $\Delta$ (see Section \ref{subsubsec:dual multicurves}) and decorate every half-edge with a permutation in $S_4$. Each permutation captures the monodromy of a loop in a system of loops associated to $\Delta$. In addition, certain cycles of the permutations receive a label. (More precisely, there is an injection from the set $\{a_1, a_2, a_3\}$ to the disjoint union of cycles of the permutations.)

Definition \ref{cons:decoration of stable tree} contains the details of this construction. There are choices to be made in the construction, but different choices will result in a $P$-decorated tree that is Hurwitz equivalent (Lemma \ref{lemma:different embedded dual trees give equivalent decorations}).
\end{sloppypar}
\end{example}

\begin{subsection}{Proof strategy} 
Let $\mathcal{H}$ be a Hurwitz space parametrising maps to $\mathbb{P}^1$ and let $P$ be its portrait. Koch \cite[Section 2]{koch2013} shows that every connected component of $\mathcal{H}$ is the quotient of a certain Teichm\"uller space $\mathcal{T}$ by the action of a \textit{liftable mapping class group} (see Section \ref{subsubsec:marked holomorphic maps from branched covers}). By Hinich and Vaintrob \cite{hinich2010augmented}, each action extends to the \textit{augmented} Teichm\"uller space $\overline{\mathcal{T}}$ (introduced in \cite{bers1974spaces}) with quotient isomorphic to a connected component of $\overline{\mathcal{H}}$.

Augmented Teichm\"uller space decomposes into locally closed, connected strata that are indexed by multicurves on a marked sphere, and the quotient map sends a stratum of $\overline{\mathcal{T}}$ to a stratum of $\overline{\mathcal{H}}$ (Proposition \ref{prop:HV10 statement}). Thus, the irreducible strata of the connected component of $\overline{\mathcal{H}}$ are indexed by orbits of multicurves under the group action. These orbits are captured by Hurwitz equivalence classes of $P$-decorated trees -- this is the content of Proposition \ref{prop:Dbar is a bijection}. The proof relies on the theory of \textit{embedded stable trees} and $P$-decorated trees developed in Sections \ref{section:embedded stable trees}, \ref{sec:decorated trees} and \ref{section:decorations from embedded stable trees}. A key step is Proposition \ref{prop:dual trees connected by simple moves}, which tells us that two embedded stable trees that are dual to a fixed multicurve are connected by a sequence of operations called \textit{braid moves}\footnote{This is reminiscent of the result in \cite{RTP} that the \textit{complex of trees} is connected, although our basic objects are different.}.

Containment of closures of strata in augmented Teichm\"uller space is given by the reverse containment of multicurves, which in turn is given by edge contraction of their dual trees, and so the containment part of Theorem \ref{thm:main theorem} follows straightforwardly. 
\end{subsection}

\begin{subsection}{Relation to tropical geometry} 
Suppose $X$ is a subvariety of the algebraic torus $(K^*)^n$, for $K$ a valued field. Classical tropical geometry explains how to associate to $X$ a piecewise affine subspace ${\op{trop}(X) \subset \mathbb{R}^n}$ -- the \textit{tropicalisation} of $X$ -- using the coordinate-wise valuation map. The tropicalisation is a `combinatorial shadow' of $X$: certain features of $X$ (such as its dimension) can be read from $\op{trop}(X)$ \cite{maclagan2015}.

There are several related constructions associating a real polyhedral space to an algebro-geometric object that we may also interpret as tropicalisations, as outlined by Abramovich, Caporaso and Payne \cite[Section 1.3]{abramovich2015tropicalization}. In increasing generality: a toric variety $X$ yields an embedded cone complex $\Sigma(X)$ called its \textit{fan}, a toroidal embedding of schemes $U \subset X$ without self-intersections yields a cone complex $\Sigma(X)$ \cite{kempf1973toroidal}, and a toroidal embedding of schemes $U \subset X$ with self-intersections yields a generalised\footnote{\textit{Generalised} means we may quotient cones by automorphisms.} cone complex $\Sigma(X)$ \cite{thuillier2007geometrie}.

The last case generalises straightforwardly to the case of a toroidal embedding $\mathcal{U} \subset \mathcal{X}$ of Deligne-Mumford stacks, also yielding a generalised cone complex $\Sigma(\mathcal{X})$ \cite{abramovich2015tropicalization}. It is natural to consider the associated extended\footnote{\textit{Extended} means cone coordinates are allowed to equal $\infty$.} generalised cone complex $\overline{\Sigma}(\mathcal{X})$, called the \textit{skeleton} of $\mathcal{X}$ -- it lives naturally inside the \textit{Berkovich analytification} of the coarse space of $\mathcal{X}$ as the image of a canonical retraction map.

The skeleton $\overline{\Sigma}(\mathcal{X})$ decomposes into extended generalised cones according to the natural stratification of $\mathcal{X}$ \cite[Proposition 6.2.6]{abramovich2015tropicalization}. For instance, the inclusion $\mathcal{M}_{g,n} \subset \overline{\mathcal{M}}_{g,n}$ is a toroidal embedding of Deligne-Mumford stacks, and the decomposition result yields a description of $\overline{\Sigma}(\overline{\mathcal{M}}_{g,n})$ in terms of $n$-marked genus-$g$ stable graphs \cite[Theorem 1.2.1]{abramovich2015tropicalization}.

The inclusion $\mathcal{H} \subset \overline{\mathcal{H}}$ is also a toroidal embedding of Deligne-Mumford stacks. Combining Theorem \ref{thm:main theorem} with the decomposition result and analysis carried out by Cavalieri, Markwig and Ranganathan for the case of $\overline{\mathcal{H}}$ \cite[Section 4]{cavalieri2016tropicalizing}, we provide an analogous description of the skeleton $\overline{\Sigma}(\overline{\mathcal{H}})$ in terms of $P$-decorated trees.

More precisely, suppose $\mathcal{H}$ is defined by the portrait $P$. A Hurwitz-equivalence class of $P$-decorated trees with $k$ edges defines an extended cone $(\mathbb{R}_{\geq 0} \cup \{\infty\})^k$. By gluing these extended cones together as prescribed by edge contraction of $P$-decorated trees, we obtain an extended cone complex $\overline{\mathcal{H}}^\text{trop}$ that is isomorphic to $\overline{\Sigma}(\overline{\mathcal{H}})$. Section \ref{subsec:Htrop definition} gives a detailed explanation.

\begin{corollary}\label{cor:skeleton theorem}
Let $\mathcal{H}$ be a Hurwitz space parametrising maps to $\mathbb{P}^1$. The extended cone complex $\overline{\mathcal{H}}^\text{trop}$ defined in Section \ref{subsec:Htrop definition} and the skeleton of the Berkovich analytification $\overline{\Sigma}(\overline{\mathcal{H}})$ are isomorphic, as extended cone complexes with integral structure.
\end{corollary}

When $\mathcal{H}$ parametrises maps to $\mathbb{P}^1$, the boundary $\overline{\mathcal{H}} \smallsetminus \mathcal{H}$ is actually a \textit{simple} normal crossings divisor, and as a result the skeleton is an extended cone complex rather than a \textit{generalised} extended cone complex.

This is a refinement of the tropicalisation given by Cavalieri, Markwig and Ranganathan for Hurwitz spaces parametrising maps to $\mathbb{P}^1$. They define an extended cone complex $\overline{\mathcal{H}}^\text{trop}_\text{CMR}$ together with a finite-to-one surjective morphism $\overline{\Sigma}(\overline{\mathcal{H}}) \rightarrow \overline{\mathcal{H}}^\text{trop}_\text{CMR}$. The failure of this morphism to be an isomorphism comes down to the fact that combinatorial admissible covers do not exactly capture the irreducible strata of $\overline{\mathcal{H}}$. We explicitly relate the two extended cone complexes $\overline{\mathcal{H}}^\text{trop} \cong \overline{\Sigma}(\overline{\mathcal{H}})$ and $\overline{\mathcal{H}}^\text{trop}_\text{CMR}$ in Section \ref{subsec:relation to CMR}.
\end{subsection}

\begin{subsection}{Application to complex dynamics}\label{subsec:application to complex dynamics}
\begin{sloppypar}

Hurwitz spaces have recently been used to study the dynamics of rational maps. If $F: \mathbb{P}^1 \rightarrow \mathbb{P}^1$ is a rational map, then its \textit{postcritical set} is the set $P_F = \bigcup_{n \geq 1} F^{\circ n}(C_F)$ of iterates of the set $C_F$ of critical (ramification) points. We say that $F$ is \textit{postcritically finite} if $P_F$ is finite. Postcritically finite rational maps are central objects of study in complex dynamics.

Suppose now that ${f:S^2 \rightarrow S^2}$ is an orientation-preserving branched cover from the sphere to itself whose postcritical set $P_f$ is finite. Such a map is called a \textit{Thurston map}. Thurston introduced these maps as topological models of postcritically finite rational maps.

The foundational result in the theory of Thurston maps is Thurston's Characterisation, which states that either a Thurston map $f$ can in a suitable sense be `promoted' to a postcritically finite rational map with the same dynamics on the postcritical set, or else $f$ has a \textit{Thurston obstruction}\footnote{This statement of Thurston's Characterisation is valid as long as $f$ has \textit{hyperbolic orbifold}. The exceptional Thurston maps with non-hyperbolic orbifold are well understood \cite[Section 9]{DH93}.}\cite{DH93}. A Thurston obstruction is a multicurve in $S^2 \smallsetminus P_f$ that lifts under $f$ to an isotopic multicurve and satisfies a certain eigenvalue condition.

\begin{sloppypar}
Hurwitz spaces arise naturally in this theory: $f$ picks out a connected component $\mathcal{H}_f$ of some Hurwitz space\footnote{The connected component $\mathcal{H}_f$ is written $\mathcal{W}_f$ in \cite{koch2013}.} (see Section \ref{subsubsec:marked holomorphic maps from branched covers}). Supposing $|P_f| = n$, the natural source and target maps ${\pi_\text{src}, \pi_\text{tgt}: \mathcal{H}_f \rightarrow \mathcal{M}_{0,n}}$ induce a multi-valued \textit{algebraic correspondence} on moduli space ${\pi_\text{src} \circ (\pi_\text{tgt})^{-1}: \mathcal{M}_{0,n} \rightrightarrows \mathcal{M}_{0,n}}$ which contains information about $f$ -- for example, if $f$ has no Thurston obstruction, then the correspondence has a fixed point \cite{koch2013}.
\end{sloppypar}

\begin{sloppypar}
Theorem \ref{thm:main theorem} has applications to the algebraic correspondence. First, note that because the source and target maps extend to ${\pi_\text{src}, \pi_\text{tgt}: \overline{\mathcal{H}}_f \rightarrow \overline{\mathcal{M}}_{0,n}}$, the correspondence extends to $\overline{\mathcal{M}}_{0,n}$. Using the $P$-decorated trees of $\overline{\mathcal{H}}_f$, we can compute the action of the correspondence on the strata of $\overline{\mathcal{M}}_{0,n}$. This is explained in Section \ref{subsec:action of piA and piB on strata}. There are immediate uses. For one, a Thurston obstruction for $f$ determines a fixed boundary stratum of $\overline{\mathcal{M}}_{0,n}$ under the correspondence -- so if no boundary stratum of $\overline{\mathcal{M}}_{0,n}$ is fixed, then $f$ cannot have a Thurston obstruction. In certain special cases, this reduces an infinite search for obstructions (of all multicurves in $S^2 \smallsetminus P_f$) to a finite computation. In addition, certain dynamical invariants of the correspondence can be computed using the boundary stratification of $\overline{\mathcal{H}}_f$ \cite{ramadas2020}.
\end{sloppypar}

Theorem \ref{thm:main theorem} also has applications to tropical Thurston theory. Ramadas develops a tropical perspective of the theory of Thurston maps, and in particular defines a \textit{tropical} correspondence whose dynamics captures important information about the dynamics of the algebraic correspondence \cite[Theorem 1.1]{ramadas2024thurston}. One difficulty in using	 this perspective, however, is the lack of an explicit description of the cone over the boundary complex of $\overline{\mathcal{H}}$, which is the `true' tropical Hurwitz space \cite[Remark 6.2]{ramadas2024thurston}. The cone complex $\mathcal{H}^\text{trop}$ that we define in Section \ref{subsec:Htrop definition} is exactly the cone over the boundary complex, addressing this difficulty. In particular, the connected components of $\mathcal{H}^\text{trop}$ are in bijection with the connected components of $\overline{\mathcal{H}}$, which is not in general true of the cone complex $\mathcal{H}^\text{trop}_\text{CMR}$. We reformulate the relevant definitions in Section \ref{subsec:tropical thurston theory}.

\begin{definition}\label{def:tropical correspondence}
The \textit{tropical moduli space correspondence} is
$$(\pi_\text{src}^\text{trop}|_{\mathcal{H}_f^\text{trop}}) \circ (\pi_\text{tgt}^\text{trop}|_{\mathcal{H}_f^\text{trop}})^{-1}: \mathcal{M}_{0,n}^\text{trop} \rightrightarrows \mathcal{M}_{0,n}^\text{trop},$$
where $\pi_\text{tgt}^\text{trop}$ and $\pi_\text{src}^\text{trop}$ are the tropical target and source maps defined in Section \ref{subsec:tropical src and tgt maps}. 
\end{definition}
\end{sloppypar}
\end{subsection}

\begin{subsection}{Related work}
Len, Ulirsch and Zakharov expand on \cite{cavalieri2016tropicalizing} and, following their recipe, outline a tropicalisation of the Abramovich-Corti-Vistoli compactification of moduli spaces of connected $G$-covers, where $G$ is an arbitrary finite group \cite[Section 3]{len2024abelian}. In the rest of their paper, they study the case where $G$ is abelian. In this paper, as in \cite{cavalieri2016tropicalizing}, we are working with the case $G=S_d$.

Brandt, Chan and Kannan \cite{brandt2024weight} also study the Abramovich-Corti-Vistoli compactification of the moduli space of $G$-covers, in the case that $G$ is a finite abelian group. They describe the stratification of these compactifications, and the combinatorial objects that they define -- admissible $G$-covers of graphs -- are closely related to our decorated trees. One key difference is that (because $G$ is abelian) there is no need for an admissible $G$-cover of graphs to keep track of the cyclic order of half-edges around each vertex, as our decorated trees must. They use their description of the stratification to compute the weight-zero compactly supported cohomology of hyperelliptic Hurwitz spaces.
\end{subsection}

\begin{subsection}{Organisation} 
\begin{description}[leftmargin=1.5em]
\item[\S \ref{sec:background}] gives background, including Teichm\"uller space, monodromy and Hurwitz spaces. 

\item[\S \ref{section:embedded stable trees}] develops the theory of \textit{embedded stable trees} and \textit{braid moves}. 

\item[\S \ref{sec:decorated trees}] defines decorated trees and Hurwitz equivalence of decorated trees.

\item[\S \ref{section:decorations from embedded stable trees}] explains the construction that assigns a decorated tree to an embedded stable tree.

\item[\S \ref{sec:decorated trees from multicurves}] explains the construction that assigns a Hurwitz equivalence class of decorated trees to a stratum of $\overline{\mathcal{H}}$, and concludes with a proof of Theorem \ref{thm:main theorem}.

\item[\S \ref{sec:strata of H under target and source}] describes the action of the target and source maps on the strata of $\overline{\mathcal{H}}$.

\item[\S \ref{sec:Htrop}] explains the connection to tropical geometry. 

\item[\S \ref{sec:appendix}] contains a technical proof from \S \ref{section:embedded stable trees}.
\end{description}
\end{subsection}

\begin{subsection}{Conventions}
We work over $\mathbb{C}$ and study compactified Hurwitz spaces as complex orbifolds, rather than as smooth Deligne-Mumford stacks, in order to apply the results of Hinich and Vaintrob \cite{hinich2010augmented}.

In keeping with tropical geometry terminology, we allow graphs in this paper to have unpaired half-edges called \textit{legs}.

We use a non-standard convention and write concatenation of paths from right-to-left, to match function composition.
\end{subsection}

\begin{subsection}{Acknowledgements} 
I am deeply grateful to my PhD supervisor, Rohini Ramadas, for suggesting this topic and for valuable feedback and guidance throughout. I would like to thank Amy Li, Hannah Markwig and Rob Silversmith for useful conversations, and John Smillie for helpful feedback on an earlier version of this paper. The author is supported by the Warwick Mathematics Institute Centre for Doctoral Training, and gratefully acknowledges funding from the University of Warwick.
\end{subsection}
\end{section}

\newpage
\begin{section}{Background}\label{sec:background}
\begin{sloppypar}

Throughout, fix a finite subset $B$ of the sphere $S^2$ with ${|B| \geq 3}$, and a finite subset $A$ of a compact connected genus-$g$ surface $X$ with $2-2g-|A| <0$. In this section, we review the definitions of three key spaces: the Teichm\"uller space $\mathcal{T}_{0,B}$, the moduli space $\mathcal{M}_{0,B}$ and the moduli space $\mathcal{M}_{g,A}$. We recall how to `augment' these spaces, yielding the augmented Teichm\"uller space $\overline{\mathcal{T}}_{0,B}$ and the Deligne-Mumford compactifications $\overline{\mathcal{M}}_{0,B}$ and $\overline{\mathcal{M}}_{g,A}$. These spaces come with natural stratifications that are indexed by multicurves, stable trees and stable graphs, respectively; we recall the definitions of these objects.

We then discuss branched covers and monodromy, before introducing the Hurwitz space $\mathcal{H}$ and the Abramovich-Corti-Vistoli compactification $\overline{\mathcal{H}}$. All surfaces are oriented, and all maps are assumed to be orientation preserving.
\end{sloppypar}

\begin{subsection}{Teichm\"uller space and moduli space}\label{subsec:teichmuller space and moduli space}

\begin{subsubsection}{The Teichm\"uller space \texorpdfstring{$\mathcal{T}_{0,B}$}{T0B}} 
\begin{sloppypar}
For our purposes, a \textit{complex structure}\footnote{This definition is equivalent to the usual atlas-of-charts definition of complex structures.} on $S^2$ is a homeomorphism ${\varphi: S^2 \xrightarrow{\cong} \mathbb{P}^1}$ up to postcomposition with a M\"obius transformation: two homeomorphisms ${\varphi: S^2 \rightarrow \mathbb{P}^1}$ and ${\varphi': S^2 \rightarrow \mathbb{P}^1}$ are equivalent complex structures if and only if there is a M\"obius transformation $\mu$ such that ${\varphi' = \mu \circ \varphi}$.

Roughly speaking, the \textit{Teichm\"uller space} $\mathcal{T}_{0,B}$ is the collection of complex structures on $S^2$ up to isotopy $\op{rel. } B$. As a set,
$$\mathcal{T}_{0,B} = \left\{ \varphi: S^2 \xlongrightarrow{\cong} \mathbb{P}^1\right\} \big/ \sim, $$
where $\varphi \sim \varphi'$ if and only if there exists a M\"obius transformation $\mu$ such that $\varphi'$ is isotopic to $\mu \circ \varphi$ $\op{rel. } B$. The Teichm\"uller space $\mathcal{T}_{0,B}$ is a simply connected complex manifold of dimension $|B|-3$.
\end{sloppypar}
\end{subsubsection}

\begin{subsubsection}{The moduli space \texorpdfstring{$\mathcal{M}_{0,B}$}{M0B}}\label{subsubsec:moduli space}
\begin{sloppypar} 
The \textit{moduli space} $\mathcal{M}_{0,B}$ is closely related. It is the collection of injections ${\iota: B \hookrightarrow \mathbb{P}^1}$ up to postcomposition with a M\"obius transformation. That is,
$$\mathcal{M}_{0,B} = \left\{\iota: B \hookrightarrow \mathbb{P}^1\right\} / \sim,$$
where ${\iota \sim \iota'}$ if and only if there is a M\"obius transformation ${\mu: \mathbb{P}^1 \rightarrow \mathbb{P}^1}$ such that ${\iota' = \mu \circ \iota}$. We call these injections \textit{markings}. The moduli space is a smooth quasiprojective variety of dimension $|B|-3$.
\end{sloppypar}

There is a straightforward map from $\mathcal{T}_{0,B}$ to $\mathcal{M}_{0,B}$ given by restricting complex structures to the set $B$:
\begin{align*}
\pi: \mathcal{T}_{0,B} &\longrightarrow \mathcal{M}_{0,B}\\
\left[\varphi\right] &\longmapsto \left[\varphi|_B\right].
\end{align*}
This map is a holomorphic covering map and realises Teichm\"uller space as the universal cover of moduli space.
\end{subsubsection}

\begin{sloppypar}
The group of deck transformations of $\pi$ can be identified with the \textit{pure mapping class group} $\op{PMod}_{0,B}$. This is the group of self-homeomorphisms of the sphere that fix $B$ pointwise, up to isotopy $\op{rel. } B$: 
$$\op{PMod}_{0,B} = \left\{h: S^2 \xlongrightarrow{\cong} S^2 \mid h(b) = b \text{ for all } b \in B\right\} / \sim$$
where ${h \sim h'}$ if and only if $h'$ is isotopic to $h$ $\op{rel. } B$. The pure mapping class group acts on Teichm\"uller space by precomposition: if ${\left[\varphi: S^2 \rightarrow \mathbb{P}^1\right]} \in \mathcal{T}_{0,B}$ and ${\left[h:S^2 \rightarrow S^2\right]} \in \op{PMod}_{0,B}$, then $\left[h\right] \cdot \left[\varphi\right] = \left[\varphi \circ h^{-1} \right].$
\end{sloppypar}

\begin{subsubsection}{The moduli space \texorpdfstring{$\mathcal{M}_{g,A}$}{MgA}}
\begin{sloppypar}
Although our focus in this paper is on the moduli space $\mathcal{M}_{0,B}$, we will need some familiarity with the more general moduli space $\mathcal{M}_{g,A}$ -- parametrising markings of genus-$g$ Riemann surfaces -- in preparation for Section \ref{sec:strata of H under target and source}. As a set,
$$\mathcal{M}_{g,A} = \left\{\iota: A \hookrightarrow \mathcal{X} \mid \mathcal{X} \text{ a genus-$g$ Riemann surface}\right\} \big/ \sim,$$
where $\iota: A \hookrightarrow \mathcal{X} \sim \iota': A \hookrightarrow \mathcal{X}'$ if and only if there is an isomorphism $\mu: \mathcal{X} \rightarrow \mathcal{X}'$ so that $\iota' = \mu \circ \iota$. The moduli space $\mathcal{M}_{g,A}$ is a smooth quasiprojective variety of dimension ${3g - 3 + |A|}$. 
\end{sloppypar}
\end{subsubsection}

\end{subsection}

\begin{subsection}{Multicurves, stable trees and stable graphs}\label{subsec:multicurves and stable trees}
We will soon `augment' $\mathcal{T}_{0,B}$ by adding in degenerate complex structures, and compactify $\mathcal{M}_{0,B}$ and $\mathcal{M}_{g,A}$ by adding in degenerate markings. The resulting spaces come with natural stratifications indexed by \textit{multicurves}, \textit{stable trees} and \textit{stable graphs}, respectively, which we now define.

\begin{subsubsection}{Multicurves} 
\begin{sloppypar} 
A simple closed curve $\gamma$ in $S^2 \smallsetminus B$ is \textit{essential} if each component of $S^2 \smallsetminus \gamma$ contains at least two points of $B$. A \textit{multicurve} in $(S^2,B)$ is a collection ${\Delta = \{\gamma_1, \ldots, \gamma_k\}}$ of pairwise disjoint, pairwise non-homotopic, essential simple closed curves in $S^2 \smallsetminus B$. 
\end{sloppypar}

We often consider multicurves up to isotopy $\op{rel. } B$: two multicurves $\Delta$ and $\Delta'$ are \textit{isotopic} if there is a bijection $\Delta \rightarrow \Delta'$ that takes each curve in $\Delta$ to an isotopic curve in $\Delta'$. We often suppress whether we are talking about a multicurve or its isotopy class, writing for instance $\Delta \subset \Delta'$ if there is an injection $\Delta \hookrightarrow \Delta'$ taking each curve in $\Delta$ to an isotopic curve in $\Delta'$. We write $\op{Mult}_{0,B}$ for the set of multicurves in $(S^2,B)$ up to isotopy.
\end{subsubsection}

\begin{subsubsection}{Stable trees}\label{subsubsec:stable trees}
In keeping with tropical geometry terminology, we allow graphs in this paper to have unpaired half-edges called \textit{legs}. A \textit{$B$-marked tree} is a connected tree $\mathscr{T}$ with exactly $|B|$ legs, together with a bijection from $B$ to the legs (referred to as the \textit{labelling}). In diagrams, we draw legs as arrows. A $B$-marked tree $\mathscr{T}$ is \textit{stable} if every vertex has valence at least three. We write $\op{Stab}_{0,B}$ for the set of isomorphism classes of $B$-marked stable trees.

\begin{figure}[ht]
\centering
\includegraphics[scale=1]{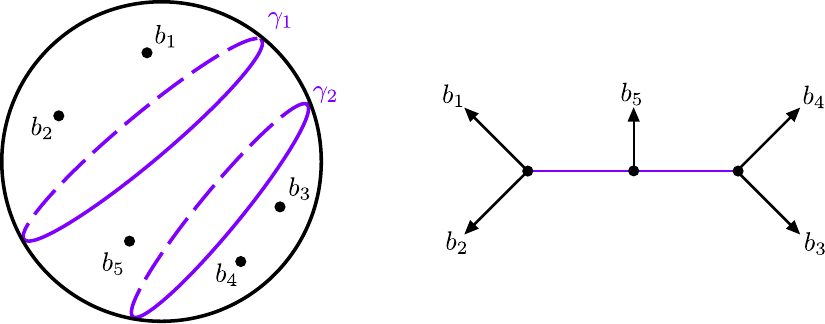}
\caption{A multicurve $\Delta = \{\gamma_1, \gamma_2\}$ in $(S^2,B)$ (left) and its dual tree $\mathscr{T}_\Delta$ (right).}
\label{fig:embedded dual tree example}
\end{figure}

Let $\Delta$ be a multicurve in $(S^2,B)$. Its \textit{dual tree} is a $B$-marked tree $\mathscr{T}_\Delta$ with one vertex for each component of $S^2 \smallsetminus \Delta$, an edge connecting two vertices if the closures of their components intersect, and a leg labelled $b$ incident with a vertex if the point $b$ is in its component. Figure \ref{fig:embedded dual tree example} gives an example. Because the curves in $\Delta$ are essential, $\mathscr{T}_\Delta$ is a stable $B$-marked tree. If $\Delta' \subset \Delta$, then $\mathscr{T}_{\Delta'}$ is obtained by contracting edges of $\mathscr{T}_\Delta$ -- specifically, the edges that correspond to the curves in $\Delta \smallsetminus \Delta'$. 
\end{subsubsection}

\begin{subsubsection}{Stable graphs}\label{subsubsec:stable graphs}
Although we deal mostly with $B$-marked stable trees, we recall a more general definition now in preparation for Section \ref{sec:strata of H under target and source}. A \textit{(vertex-weighted) $A$-marked graph} is a connected graph $\mathscr{G}$ with exactly $|A|$ legs, together with an injection from $A$ to the legs, and an assignment of a nonnegative integer weight to each vertex. The \textit{genus} of $\mathscr{G}$ is the sum of the vertex weights plus the topological genus of the graph. The graph $\mathscr{G}$ is \textit{stable} if all vertices of weight zero have valence at least three. An \textit{isomorphism} of $A$-marked graphs must preserve the labelling. We write $\op{Stab}_{g,A}$ for the set of isomorphism classes of $A$-marked stable graphs of genus $g$. (In our discussion of $B$-marked stable trees, the vertex weights are all implicitly set to zero.)

Given a non-stable $A$-marked genus-$g$ graph $\mathscr{G}$, there exists a canonical stable $A$-marked \mbox{genus-$g$} graph $\mathscr{G}'$ called the \textit{stabilisation}: weight-zero vertices of $\mathscr{G}$ with valence one are contracted into their neighbouring vertex, and weight-zero vertices with valence two are `absorbed' into an edge or leg. In this way, a vertex of $\mathscr{G}'$ can be identified with a vertex of $\mathscr{G}$, and a half-edge of $\mathscr{G}'$ can be identified with a \textit{path} of half-edges of $\mathscr{G}$, as in Figure \ref{fig:graph stabilisation example}.

\begin{figure}[ht]
\centering
\includegraphics[width=1\linewidth]{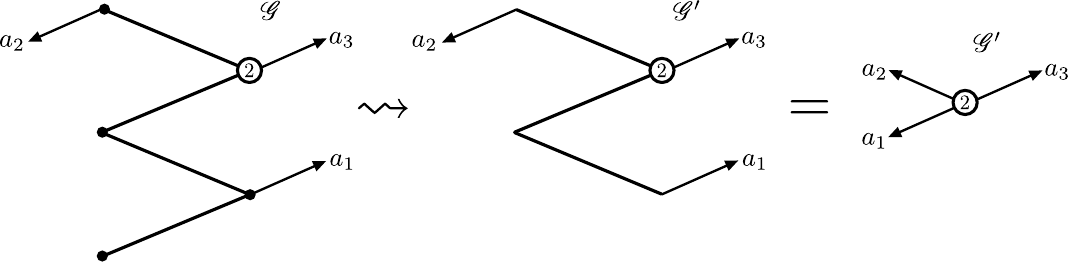}
\caption{Let $A = \{a_1,a_2,a_3\}$ and $g=2$. The non-stable $A$-marked \mbox{genus-$g$} graph $\mathscr{G}$ stabilises to the stable $A$-marked genus-$g$ graph $\mathscr{G}'$. Unlabelled vertices have weight zero.}
\label{fig:graph stabilisation example}
\end{figure}

\end{subsubsection}
\end{subsection}

\begin{subsection}{Augmented Teichm\"uller space and the Deligne-Mumford compactification}\label{subsec:augteich and DM compactification}
\begin{subsubsection}{Augmented Teichm\"uller space \texorpdfstring{$\overline{\mathcal{T}}_{0,B}$}{{\=T}0B}}\label{subsubsec:augteich}
Bers introduced the augmented Teichm\"uller space $\overline{\mathcal{T}}_{0,B}$, which is obtained by adding points to Teichm\"uller space that correspond to degenerate complex structures in which some multicurve is collapsed to a point \cite{bers1974spaces}.

\begin{sloppypar}
A \textit{nodal complex structure} on $S^2$ is a continuous `pinching' map $\varphi: S^2 \rightarrow \mathcal{X}$ to a compact connected nodal Riemann surface of arithmetic genus 0, where the preimage of the set of nodes of $\mathcal{X}$ is a multicurve $\Delta$ in $(S^2,B)$, and the restriction of $\varphi$ to $S^2 \smallsetminus \Delta$ is a homeomorphism onto the non-nodal points of $\mathcal{X}$. The \textit{augmented Teichm\"uller space} is the collection of nodal complex structures up to isotopy $\op{rel. } B$:
$$\overline{\mathcal{T}}_{0,B} = \left\{\varphi: S^2 \rightarrow \mathcal{X} \text{ a nodal complex structure} \right\} \big/ \sim$$
where ${\varphi: S^2 \rightarrow \mathcal{X}} \sim {\varphi': S^2 \rightarrow \mathcal{X}'}$ if and only if there exists an isomorphism ${\mu: \mathcal{X} \rightarrow \mathcal{X}'}$ such that $\varphi'$ and $\mu \circ \varphi$ are isotopic rel. $B$. The augmented Teichm\"uller space is not a manifold, and is not compact or even locally compact.
\end{sloppypar}

\begin{sloppypar}
$\overline{\mathcal{T}}_{0,B}$ decomposes into locally closed, connected strata that are indexed by $\op{Mult}_{0,B}$: the point ${\left[\varphi:S^2 \rightarrow \mathcal{X}\right]}$ is in the stratum (indexed by) $\Delta$ if $\varphi$ collapses precisely a multicurve isotopic to $\Delta$. Containment of strata is given by reverse containment of multicurves: the stratum $\Delta$ is contained in the closure of the stratum $\Delta'$ if and only if $\Delta' \subset \Delta$. Observe that the stratum $\Delta = \emptyset$ is the non-augmented Teichm\"uller space. 
\end{sloppypar}
\end{subsubsection}

\begin{subsubsection}{The Deligne-Mumford compactification \texorpdfstring{$\overline{\mathcal{M}}_{0,B}$}{{\=M}0B}}
Similarly, there is a compactification $\overline{\mathcal{M}}_{0,B}$ of moduli space obtained by adding in degenerate markings.

\begin{sloppypar}
A \textit{nodal marking} is an injection ${\iota: B \hookrightarrow \mathcal{X}}$ to a compact connected nodal Riemann surface of arithmetic genus 0, such that $\iota(B)$ does not intersect the set of nodes of $\mathcal{X}$. The \textit{dual tree} of $\iota$ is a $B$-marked tree with one vertex for each irreducible component of $\mathcal{X}$, an edge connecting vertices if their components intersect at a node, and a leg labelled $b \in B$ incident with a vertex if the point $\iota(b)$ is in its component. The nodal marking $\iota$ is \textit{stable} if its dual tree is stable.

The \textit{Deligne-Mumford compactification} $\overline{\mathcal{M}}_{0,B}$ is the collection of stable nodal markings up to isomorphism:
$$\overline{\mathcal{M}}_{0,B} = \left\{ \iota: B \hookrightarrow \mathcal{X} \text{ a stable nodal marking } \right\} \big/ \sim$$
where ${\iota: B \hookrightarrow \mathcal{X}} \sim {\iota': B \hookrightarrow \mathcal{X}'}$ if and only if there exists an isomorphism ${\mu: \mathcal{X} \rightarrow \mathcal{X}'}$ such that ${\iota' = \mu \circ \iota}$. The compactification is a smooth projective variety. 
\end{sloppypar}

\begin{sloppypar}
$\overline{\mathcal{M}}_{0,B}$ decomposes into locally closed, irreducible strata that are indexed by $\op{Stab}_{0,B}$: the point ${[\iota: B \hookrightarrow \mathcal{X}]}$ is in the stratum (indexed by) $\mathscr{T}$ if its dual tree is isomorphic to $\mathscr{T}$. Containment of strata is given by edge contraction of $B$-marked stable trees: the stratum $\mathscr{T}$ is contained in the closure of the stratum $\mathscr{T}'$ if and only if $\mathscr{T}$ contracts along edges to $\mathscr{T}'$. Observe that the stratum of the one-vertex tree $\mathscr{T} = \mathscr{T}_{\emptyset}$ is the non-compactified moduli space.
\end{sloppypar}

\begin{sloppypar}
The action of $\op{PMod}_{0,B}$ on $\mathcal{T}_{0,B}$ extends straightforwardly to $\overline{\mathcal{T}}_{0,B}$. The next proposition tells us that the quotient (when given the correct complex analytic structure) is $\overline{\mathcal{M}}_{0,B}$ \cite{hubbard2014analytic}, and that the irreducible strata of $\overline{\mathcal{M}}_{0,B}$ can be interpreted as orbits of multicurves under the action of $\op{PMod}_{0,B}$.
\end{sloppypar}

\begin{proposition}\label{prop:strata of moduli space are orbits of multicurves}
There is an isomorphism $\overline{\mathcal{T}}_{0,B} / \op{PMod}_{0,B} \cong \overline{\mathcal{M}}_{0,B}$ of complex analytic spaces. This isomorphism identifies the orbit $\op{PMod}_{0,B}(\Delta)$ of the stratum $\Delta$ of $\overline{\mathcal{T}}_{0,B}$ with the stratum $\mathscr{T}_\Delta$ of $\overline{\mathcal{M}}_{0,B}$. Furthermore, containment of closures of strata descends from the reverse containment of multicurves.
\end{proposition}

We will prove Theorem \ref{thm:main theorem} in analogy with this proposition: every connected component of a compactified Hurwitz space is a quotient of $\overline{\mathcal{T}}_{0,B}$ by a \textit{subgroup} $\op{LMod}_f \leq \op{PMod}_{0,B}$, and this identifies the irreducible strata of the connected component with orbits of multicurves under the action of $\op{LMod}_f$ (Proposition \ref{prop:HV10 statement}). Just as the orbits $\op{PMod}_{0,B}(\Delta)$ are indexed by stable trees, the orbits $\op{LMod}_f(\Delta)$ are indexed by Hurwitz equivalence classes of decorated trees. 
\end{subsubsection}

\begin{subsubsection}{The Deligne-Mumford compactification \texorpdfstring{$\overline{\mathcal{M}}_{g,A}$}{{\=M}gA}}\label{subsubsec:MgAbar}
The Deligne-Mumford compactification $\overline{\mathcal{M}}_{g,A}$ is similar. A \textit{genus-$g$ nodal marking} is an injection $\iota: A \hookrightarrow \mathcal{X}$ to a compact nodal Riemann surface of arithmetic genus $g$, such that $\iota(A)$ does not intersect the nodes of $\mathcal{X}$. It is \textit{stable} if its $A$-marked dual graph (with each vertex weighted by the genus of the normalisation of the corresponding component) is stable. The Deligne-Mumford compactification $\overline{\mathcal{M}}_{g,A}$ is the collection of genus-$g$ stable nodal markings, up to isomorphism:
$$\overline{\mathcal{M}}_{g,A} = \left\{ \iota: A \hookrightarrow \mathcal{X} \text{ a genus-$g$ stable nodal marking } \right\} \big/ \sim$$
where ${\iota: A \hookrightarrow \mathcal{X}} \sim {\iota': A \hookrightarrow \mathcal{X}'}$ if and only if there exists an isomorphism ${\mu: \mathcal{X} \rightarrow \mathcal{X}'}$ such that ${\iota' = \mu \circ \iota}$. The compactification is a projective variety.

$\overline{\mathcal{M}}_{g,A}$ decomposes into locally closed, irreducible strata that are indexed by $\op{Stab}_{g,A}$: the point $[\iota: A \hookrightarrow \mathcal{X}]$ is in the stratum (indexed by) $\mathscr{G} \in \op{Stab}_{g,A}$ if the dual graph of $\iota$ is isomorphic to $\mathscr{G}$. Containment of closures of strata is given by edge contraction.

Given a non-stable genus-$g$ nodal marking $\iota: A \hookrightarrow \mathcal{X}$, there exists a canonical stable genus-$g$ nodal marking $\iota': A \hookrightarrow \mathcal{X}'$ called the \textit{stabilisation}. It is obtained by contracting genus-$0$ components that contain fewer than three special points (meaning nodes or marked points). The dual graph $\mathscr{G}'$ of $\iota'$ is the stabilisation of the dual graph $\mathscr{G}$ of $\iota$. 
\end{subsubsection}
\end{subsection}

\begin{subsection}{Branched covers of surfaces}
\begin{sloppypar}
Suppose $f: X \rightarrow S^2$ is a continuous map so that $f^{-1}(B)$ is finite,
$$f: X \smallsetminus f^{-1}(B) \longrightarrow S^2 \smallsetminus B$$
is a covering map of degree $d$, and $A \subset f^{-1}(B)$. We call the map of pairs ${f: (X,A) \rightarrow (S^2,B)}$ a \textit{branched cover of degree $d$} to the sphere.

Away from $f^{-1}(B)$, the map $f$ is a local homeomorphism. At a point $x \in f^{-1}(B)$, however, $f$ can be given local coordinates (sending $x$ and $f(x)$ to 0) so that $f$ is ${z \mapsto z^k}$ for some integer ${k \geq 1}$. The number $k$ is well-defined and is called the \textit{local degree} (or \textit{ramification}) of $f$ at $x$, denoted $\op{rm}_f(x)$.

For $y \in B$, we write ${\op{br}_f(y) = \left(\op{rm}_f(x)\right)_{x \in f^{-1}(y)}}$ for the multi-set of local degrees of the points that map to $y$. This multi-set is a partition of the degree $d$ and is called the \textit{branch profile} of $y$. 
\end{sloppypar}

\begin{subsubsection}{Hurwitz equivalence of branched covers}
We now define a notion of equivalence for branched covers. It will be central later on in describing the connected components of a Hurwitz space.

\begin{definition} Let $f_1, f_2: (X, A) \rightarrow (S^2,B)$ be branched covers. We say that $f_1$ and $f_2$ are \textit{$(A,B)$-Hurwitz equivalent (via $h$, $\widetilde{h}$)} if there are homeomorphisms 
$$h: (S^2,B) \longrightarrow (S^2,B)\quad\text{and}\quad \widetilde{h}: (X,A) \longrightarrow (X,A)$$
so that $h$ fixes $B$ pointwise, $\widetilde{h}$ fixes $A$ pointwise, and the following diagram commutes
\begin{equation}\label{cd:hurwitz equivalence of branched covers}
\begin{tikzcd}
(X,A) \arrow[d,"f_1"] \arrow[r,"\widetilde{h}"] & (X,A) \arrow[d,"f_2"]\\
(S^2,B) \arrow[r, "h"] & (S^2,B)
\end{tikzcd}
\end{equation}
\end{definition}
\end{subsubsection}

\begin{subsubsection}{The liftable mapping class group} 
\begin{sloppypar}
We say that a homeomorphism ${h:(S^2,B) \rightarrow (S^2,B)}$ fixing $B$ pointwise is \textit{liftable} under a branched cover $f: (X,A) \rightarrow (S^2,B)$ if there exists a homeomorphism ${\widetilde{h}:(X,A) \rightarrow (X,A)}$ fixing $A$ pointwise such that ${h \circ f = f \circ \widetilde{h}}$. That is, if $h$ and $\widetilde{h}$ give an $(A,B)$-Hurwitz equivalence of $f$ with itself.

The set of liftable homeomorphisms (for a fixed $f$) descends to a finite-index subgroup $\op{LMod}_f \leq \op{PMod}_{0,B}$ called the \textit{liftable mapping class group} of $f$ \cite[Corollary 2.2]{koch2013}.
\end{sloppypar}
\end{subsubsection}
\end{subsection}

\begin{subsection}{Monodromy}
In order to give a combinatorial description of the strata of Hurwitz spaces, we will want to extract some combinatorics from branched covers. The concept of \textit{monodromy} gives us a meaningful way of doing this.

\begin{subsubsection}{Monodromy representations} 
\begin{sloppypar}
Fix a branched cover ${f:(X,A) \rightarrow (S^2,B)}$ of degree $d$. Choose a basepoint ${y \in S^2 \smallsetminus B}$. By definition, $f$ is a covering map away from the set $B$, and so $f^{-1}(y)$ contains exactly $d$ points. Choose a labelling of the preimage
$$l: f^{-1}(y) \longrightarrow \{y_1, \ldots, y_d\}.$$
If $\gamma$ is a loop in $S^2 \smallsetminus B$ based at $y$, then for any ${1 \leq i \leq d}$, $\gamma$ lifts to a path in ${X \smallsetminus f^{-1}(B)}$ from $y_i$ to some $y_j$ (by the path-lifting property of covering spaces). Setting $\sigma(i) = j$, we obtain a permutation $\sigma \in S_d$ called the \textit{monodromy permutation} of $\gamma$.

The monodromy permutation of $\gamma$ depends only on the homotopy class of $\gamma$. Thus, choosing a basepoint $y \in S^2 \smallsetminus B$ and a labelling ${l: f^{-1}(y) \rightarrow \{y_1, \ldots, y_d\}}$, we obtain a homomorphism
$$\Sigma: \pi_1(S^2 \smallsetminus B, y) \longrightarrow S_d$$
called the \textit{monodromy representation} of $f$. The image of the monodromy representation is called the \textit{monodromy group}.
\end{sloppypar}
\end{subsubsection}

The monodromy representation captures certain features of $f$, as suggested by the next two propositions. The proofs are omitted.

\begin{proposition}\label{prop:transitive monodromy group implies connected source}
With the notation above, the surface $X$ is connected if and only if the monodromy group acts transitively on $\{1, \ldots, d\}$. 
\end{proposition}

\begin{subsubsection}{Keyhole loops}
There are special loops whose monodromy permutations are particularly important for us. A \textit{keyhole loop} based at $y$ is a simple loop that bounds (on its left) a disk containing a single point of $B$. The next proposition asserts that the monodromy permutations of keyhole loops contain information about the branching of $f$. Figure \ref{fig:monodromy explainer} gives the idea of the proof.

\begin{proposition}\label{prop:keyhole loop monodromy}
Let $f: (X, A) \rightarrow (S^2, B)$ be a branched cover of degree $d$, and choose a basepoint $y \in S^2 \smallsetminus B$ and a labelling $l: f^{-1}(y) \rightarrow \{y_1, \ldots, y_d\}$. Let $\gamma$ be a keyhole loop about $b \in B$ with monodromy permutation $\Sigma(\gamma)$. Then
\begin{enumerate}[label=(\roman*)]
\item there is a canonical bijection 
$$\mathit{cyc}_\gamma: f^{-1}(b) \longrightarrow \op{Cycles}(\Sigma(\gamma))$$
between $f^{-1}(b)$ and the disjoint cycles of $\Sigma(\gamma)$ (including cycles of length one) and
\item under this bijection, $\mathit{cyc}_\gamma(a)$ is a cycle of length $\op{rm}_f(a)$.
\end{enumerate}
\end{proposition}

\begin{figure}[ht]
\centering
\includegraphics[scale=1]{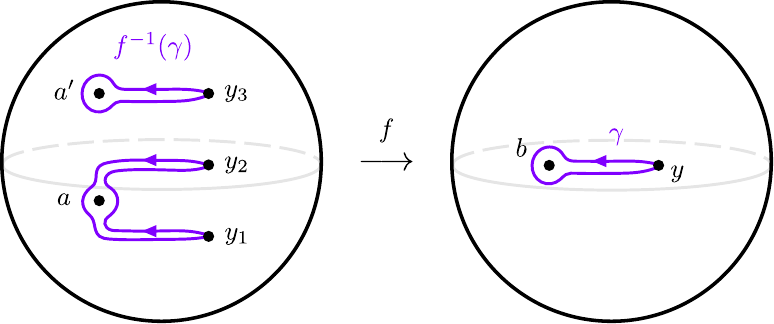}
\caption{\sloppy $f$ is a degree-3 branched cover with $f^{-1}(b) = \{a,a'\}$, where $\op{rm}_f(a) = 2$ and $\op{rm}_f(a') = 1$. The loop $\gamma$ has monodromy permutation ${\Sigma(\gamma) = (1\,2)(3)}$, and ${\mathit cyc}_\gamma(a) = (1\,2)$ and ${\mathit cyc}_\gamma(a') = (3)$.}
\label{fig:monodromy explainer}
\end{figure}
\end{subsubsection}
\end{subsection}

\begin{subsection}{Portraits}\label{subsec:portraits}
A Hurwitz space parametrises holomorphic maps of Riemann surfaces with marked points that satisfy prescribed branching conditions. The presribed branching conditions are encoded in a \textit{portrait}\footnote{These are called \textit{static} portraits in \cite{koch2013}, to distinguish them from \textit{dynamical} portraits.}, defined as follows (modifying notation from \cite{ramadas2017}).

\begin{definition}\label{def:portrait}
A \textit{portrait} is the data $P = (g,d,B,A, \phi:A \rightarrow B,\op{br},\op{rm})$, where
\begin{itemize}
\item $g$ is a nonnegative integer (the \textit{source genus}\footnote{In this paper, we only consider maps whose target surface is the sphere, and so our portraits will only record the genus of the source.})
\item $d$ is a positive integer (the \textit{degree})
\item $B$ is a finite set with $|B| \geq 3$ (the \textit{target marks})
\item $A$ is a finite set with $2 - 2g - |A| < 0$ (the \textit{source marks})
\item $\phi: A \rightarrow B$ (the \textit{function} on the source marks)
\item $\op{br}: B \rightarrow \{\text{partitions of } d\}$ (the \textit{branch profiles} of the target marks)
\item $\op{rm}: A \rightarrow \Z_{\geq 0}$ (the \textit{ramifications} (local degrees) at the source marks)
\end{itemize}
satisfying the following two conditions:
\begin{enumerate}
\item $\sum_{b \in B} (d - \op{len}(\op{br}(b))) - 2g = 2d -2$ (the \textit{Riemann-Hurwitz formula})
\item for each $b \in B$, $(\op{rm}(a))_{a \in \phi^{-1}(b)}$ is a sub-multiset of $\op{br}(b)$
\end{enumerate}
\end{definition}

Portraits are conveniently represented by directed graphs with certain edge and vertex labels. For instance, the directed graph
\[
\begin{tikzcd}[column sep=small, row sep=12mm, column sep=2.7mm, ampersand replacement=\&]
\overset{\displaystyle a_2}{\critpt}\arrow[dr,"\scriptstyle 2" left, line width=0.5mm] 
\& \&
\stdpt\arrow[dl,"\scriptstyle 1" right, line width=0.5mm]
\& \& ~\&
\overset{\displaystyle a_3}{\critpt}\arrow[dr, "\scriptstyle 2" left, line width=0.5mm] 
\& \&
\stdpt\arrow[dl,"\scriptstyle 1" right, line width=0.5mm]
\& \& ~\&
\critpt\arrow[dr,"\scriptstyle 2" left, line width=0.5mm]
\& \&
\overset{\displaystyle a_4}{\stdpt}\arrow[dl,"\scriptstyle 1" right, line width=0.5mm]
\& \& ~\&
\critpt\arrow[dr,"\scriptstyle 2" left, line width=0.5mm]
\& \&
\overset{\displaystyle a_1}{\stdpt}\arrow[dl,"\scriptstyle 1" right, line width=0.5mm]
\\
\&
\underset{\displaystyle b_1}{\stdpt} 
\& \& \& \& \&
\underset{\displaystyle b_2}{\stdpt} 
\& \& \& \& \&
\underset{\displaystyle b_3}{\stdpt} 
\& \& \& \& \&
\underset{\displaystyle b_4}{\stdpt} 
\&
\end{tikzcd}
\]
fixes the data $d=3$, $B=\{b_1,b_2,b_3,b_4\}$, $A= \{a_1, a_2,a_3,a_4\}$,
\begin{equation*}
\begin{gathered}
\phi(a_1) = b_4,\enskip \phi(a_2) = b_1,\enskip \phi(a_3) = b_2,\enskip \phi(a_4) = b_3,\\
\op{br}(b_1) = \op{br}(b_2) =\op{br}(b_3) =\op{br}(b_4) = (2,1),\\
\op{rm}(a_1) = \op{rm}(a_4) = 1,\enskip \op{rm}(a_2) = \op{rm}(a_3) = 2.
\end{gathered}
\end{equation*}
Circled vertices signify ramification points. The source genus $g$ can be worked out from the Riemann-Hurwitz formula: $g = 0$.

\begin{sloppypar}
A branched cover ${f:(X,A) \rightarrow (S^2,B)}$ defines a portrait ${P(f) = (g, d, B, A, \phi, \op{br}, \op{rm})}$ where $g$ is the genus of $X$, $d$ is the degree of $f$, ${\phi = f|_A}$, ${\op{br} = \op{br}_f|_B}$ and ${\op{rm} = \op{rm}_f|_A}$. We say that $f$ \textit{realises} the portrait $P$ if ${P(f) = P}$. Note that if $f' :(X,A) \rightarrow (S^2,B)$ is an $(A,B)$-Hurwitz equivalent branched cover, then $P(f) = P(f')$. 
\end{sloppypar}
\end{subsection}

\begin{subsection}{Hurwitz spaces} 
\begin{sloppypar}
Let ${F: \mathcal{X} \rightarrow \mathbb{P}^1}$ be a holomorphic map of Riemann surfaces, and let ${\iota_A: A \hookrightarrow \mathcal{X}}$ and ${\iota_B: B \hookrightarrow \mathbb{P}^1}$ be injections. We call the triple ${(F,\iota_A,\iota_B)}$ a \textit{marked holomorphic map} to $\mathbb{P}^1$.

Now fix a portrait ${P = (g,d,B,A,\phi,\op{br},\op{rm})}$. We say that a marked holomorphic map ${(F: \mathcal{X} \rightarrow \mathbb{P}^1, \iota_A: A \hookrightarrow \mathcal{X}, \iota_B: B \hookrightarrow \mathbb{P}^1)}$ \textit{realises} the portrait $P$ if it `fits into' the portrait, meaning $\mathcal{X}$ is a Riemann surface of genus $g$, $F$ has degree $d$, ${\iota_B^{-1} \circ F \circ \iota_A = \phi}$ ($F$ maps the marked points as prescribed by $\phi$), ${\op{br}_{F} \circ\, \iota_B = \op{br}}$ (the branch profiles match), and ${\op{rm}_F \circ \, \iota_A = \op{rm}}$ (the local degrees match).
\end{sloppypar}

The \textit{Hurwitz space} associated to $P$ is the collection $\mathcal{H} = \mathcal{H}(P)$ of marked holomorphic maps that realise $P$, up to pre- and post-composition by isomorphisms. As a set,
$$\mathcal{H} = \left\{\left(F: \mathcal{X} \rightarrow \mathbb{P}^1,\, \iota_A: A \hookrightarrow \mathcal{X},\, \iota_B: B \hookrightarrow \mathbb{P}^1\right) \text{ realising } P \right\} \big/ \sim,$$
where $(F: \mathcal{X} \rightarrow \mathbb{P}^1,\, \iota_A,\, \iota_B) \sim (F': \mathcal{X}' \rightarrow \mathbb{P}^1,\, \iota_A',\, \iota_B')$ if and only if there exists an isomorphism $\nu: \mathcal{X} \rightarrow \mathcal{X}'$ and a M\"obius transformation $\mu: \mathbb{P}^1 \rightarrow \mathbb{P}^1$ such that 
$$F' = \mu \circ F \circ \nu^{-1},\quad \iota_A' = \nu \circ \iota_A \quad\text{and}\quad \iota_B' = \mu \circ \iota_B.$$ 
The Hurwitz space $\mathcal{H}$ is a complex manifold\footnote{A Hurwitz space is in general a complex \textit{orbifold}, but there are no orbifold singularities when the Hurwitz space parametrises maps to $\mathbb{P}^1$.} of dimension $|B|-3$, if it is nonempty.

\begin{subsubsection}{The target and source maps}
The Hurwitz space $\mathcal{H}$ comes equipped with a forgetful target map and a forgetful source map. The target map is a holomorphic covering map $\pi_B: \mathcal{H} \rightarrow \mathcal{M}_{0,B},\, [(F, \iota_A, \iota_B)] \mapsto [\iota_B]$. The source map is holomorphic but not necessarily surjective, $\pi_A: \mathcal{H} \rightarrow \mathcal{M}_{g,A},\, [(F, \iota_A, \iota_B)] \mapsto [\iota_A]$.
\end{subsubsection}

\begin{example}\label{example:Hurwitz space example d=3}
Let $\mathcal{H}$ be the Hurwitz space associated to the degree-3 portrait in Section \ref{subsec:portraits}. A point of $\mathcal{H}$ is a triple $(F: \mathbb{P}^1 \rightarrow \mathbb{P}^1,\, \iota_A: A \hookrightarrow \mathbb{P}^1,\, \iota_B: B \hookrightarrow \mathbb{P}^1)$, up to pre- and post-composition by M\"obius transformations. The rational map
$$F(z)= \frac{z^3 + cz^2}{(3+2c)z - (2+c)}$$
with markings
$$\enskip \iota_A(a_1) = \frac{c}{2c+3},\enskip \iota_A(a_2) = 0,\enskip \iota_A(a_3) = 1,\enskip \iota_A(a_4) = \frac{2+c}{3+2c}$$
and
$$\iota_B(b_1) = 0,\enskip \iota_B(b_2) = 1,\enskip \iota_B(b_3) = \infty,\enskip \iota_B(b_4) = -\frac{c^3(c+2)}{(2c+3)^3}$$
gives a point $(F,\iota_A,\iota_B)$ of the Hurwitz space for $c \in \mathbb{P}^1 \smallsetminus \{0, -2, -3, -1, -\frac{3}{2}, \infty\}$ -- we must exclude exactly these six values to avoid collisions between the marked points.

\[
\begin{tikzcd}[column sep=small, row sep=12mm, column sep=2.7mm, ampersand replacement=\&]
\overset{\displaystyle 0}{\critpt}\arrow[dr,"\scriptstyle 2" left, line width=0.5mm] 
\& \&
\stdpt\arrow[dl,"\scriptstyle 1" right, line width=0.5mm]
\& \& ~\&
\overset{\displaystyle 1}{\critpt}\arrow[dr, "\scriptstyle 2" left, line width=0.5mm] 
\& \&
\stdpt\arrow[dl,"\scriptstyle 1" right, line width=0.5mm]
\& \& ~\&
\critpt\arrow[dr,"\scriptstyle 2" left, line width=0.5mm]
\& \&
\hspace{-2.5mm}
\overset{\scriptstyle \frac{2+c}{3+2c}}{\stdpt}\arrow[dl,"\scriptstyle 1" right, line width=0.5mm]
\hspace{-2.5mm}
\& \& ~\&
\critpt\arrow[dr,"\scriptstyle 2" left, line width=0.5mm]
\& \&
\hspace{-2.5mm}
\overset{\scriptstyle \frac{c}{2c+3}}{\stdpt}\arrow[dl,"\scriptstyle 1" right, line width=0.5mm]
\hspace{-2.5mm}
\\
\&
\underset{\displaystyle 0}{\stdpt} 
\& \& \& \& \&
\underset{\displaystyle 1}{\stdpt} 
\& \& \& \& \&
\underset{\displaystyle \infty}{\stdpt} 
\& \& \& \& \&
\hspace{-3mm}
\underset{\scriptstyle -\frac{c^3(c+2)}{(2c+3)^3}}{\stdpt} 
\hspace{-3mm}
\&
\end{tikzcd}
\] 
This is in fact a parametrisation of the Hurwitz space: $\mathcal{H} \cong \mathbb{P}^1 \smallsetminus \{0,-1,-2,-3,-3/2,\infty\}$.
\end{example}

\begin{subsubsection}{Marked holomorphic maps from branched covers}\label{subsubsec:marked holomorphic maps from branched covers}
\begin{sloppypar}
Suppose ${f: (X,A) \rightarrow (S^2,B)}$ is a branched cover and $P=P(f)$ is its portrait. Then $f$ picks out a connected component $\mathcal{H}_f$ of the Hurwitz space $\mathcal{H}(P)$ as follows. If $\varphi: S^2 \rightarrow \mathbb{P}^1$ is a complex structure, then by the Riemann Existence Theorem (see \cite[Theorem 6.2.2]{cavalieri2016}) there exists a complex structure $\psi: X \xrightarrow{\cong} \mathcal{X}$ on the source surface so that $f$ becomes a marked holomorphic map ${(F = \varphi \circ f \circ \psi^{-1}: \mathcal{X} \rightarrow \mathbb{P}^1,\, \iota_A = \psi|_A,\, \iota_B = \varphi|_B)}$ in $\mathcal{H}(P)$. This construction defines a holomorphic covering map to a connected component of $\mathcal{H}(P)$ that we denote by $\mathcal{H}_f$:
\begin{align*}
\pi_f: \mathcal{T}_{0,B} &\longrightarrow \mathcal{H}_f\\
\left[\varphi\right] &\longmapsto \left[\left(F: \mathcal{X} \rightarrow \mathbb{P}^1,\, \psi|_A,\, \varphi|_B\right)\right].
\end{align*}
Two points of $\mathcal{T}_{0,B}$ induce equivalent marked holomorphic maps if and only if they differ by an element of $\op{LMod}_f$, whence $\mathcal{T}_{0,B} / \op{LMod}_f \cong \mathcal{H}_f$; \cite[Theorem 2.6]{koch2013} proves this for the case $X=S^2$ and the proof adapts straightforwardly to the case of general $X$.	

Every connected component of a Hurwitz space can be obtained as such a quotient. Furthermore, $\mathcal{H}_f = \mathcal{H}_{f'}$ for branched covers $f, f': (X,A) \rightarrow (S^2,B)$ if and only if $f$ and $f'$ are $(A,B)$-Hurwitz equivalent \cite[Corollary 2.7]{koch2013}, in which case (slightly generalising \cite[Lemma 2.5]{koch2013}), we have the following lemma.

\begin{lemma}\label{lemma:pif and pifprime}
If $f$ and $f'$ are $(A,B)$-Hurwitz equivalent via homeomorphisms $h, \widetilde{h}$, then $\pi_f = \pi_{f'} \circ h: \mathcal{T}_{0,B} \rightarrow \mathcal{H}_f = \mathcal{H}_{f'}$ (with $h$ acting on $\mathcal{T}_{0,B}$ as explained in Section \ref{subsubsec:moduli space}).
\end{lemma}

\end{sloppypar}
\end{subsubsection}

\begin{subsubsection}{Connected components of Hurwitz spaces}
For a general portrait, the associated Hurwitz space can have multiple connected components, or none\footnote{The problem of classifying the portraits whose associated Hurwitz space is empty is known as the \textit{Hurwitz existence problem} and dates back to \cite{hurwitz1891}; see \cite{edmonds1984}.}. By the previous section, the connected components are in bijection with $(A,B)$-Hurwitz equivalence classes of branched covers that realise the portrait.

\begin{proposition}\label{prop:components of Hurwitz space}
\begin{sloppypar}
Let $X$ be a compact surface of genus $g$, and let $P$ be a portrait with source genus $g$, source marks $A \subset X$ and target marks $B \subset S^2$. Let $\mathcal{H}$ be the associated Hurwitz space. Then $\mathcal{H} = \bigsqcup_{[f]} \mathcal{H}_{f},$ where the disjoint union ranges over $(A,B)$-Hurwitz equivalence classes of branched covers ${f:(X,A) \rightarrow (S^2,B)}$ that realise $P$.
\end{sloppypar}
\end{proposition}
\end{subsubsection}
\end{subsection}

\begin{subsection}{Compactifying Hurwitz spaces}\label{subsec:HV10}
Let $\mathcal{H}$ be a Hurwitz space with defining portrait $P$. Harris and Mumford give a compactification $\overline{\mathcal{H}}^\text{HM}$ by a space of \textit{admissible covers} \cite{harris1982kodaira}, which are maps of nodal Riemann surfaces. The space $\overline{\mathcal{H}}^\text{HM}$ is not normal, but its normalisation -- which we denote by $\overline{\mathcal{H}}$ -- is smooth. Abramovich, Corti and Vistoli reinterpret $\overline{\mathcal{H}}$ as a complex orbifold that parametrises \textit{balanced twisted covers} \cite{abramovich2003twisted}.

The target and source maps $\pi_B$ and $\pi_A$ extend to the boundaries of $\overline{\mathcal{H}}^\text{HM}$ and $\overline{\mathcal{H}}$: $\pi_B$ takes an admissible cover or a balanced twisted cover to the stable nodal marking in $\overline{\mathcal{M}}_{0,B}$ of its target, and $\pi_A$ takes an admissible cover or balanced twisted cover to the nodal marking in $\overline{\mathcal{M}}_{g,A}$ of its source (stabilised, if not already stable -- see Section \ref{subsubsec:MgAbar}).

The boundary of $\overline{\mathcal{H}}$ is a simple normal crossings divisor, and this endows $\overline{\mathcal{H}}$ with a natural stratification: the codimension-$k$ strata are connected components of $k$-fold intersections of divisors. The stratification of $\overline{\mathcal{H}}$ admits a description that will be useful in the proof below: an irreducible stratum of $\overline{\mathcal{H}}$ is a connected component of the preimage of an irreducible stratum of $\overline{\mathcal{M}}_{0,B}$, under $\pi_B$.

Suppose $f$ is a branched cover that realises $P$, and let $\overline{\mathcal{H}}_f$ be the connected component of $\overline{\mathcal{H}}$ containing $\mathcal{H}_f$. The next proposition, in analogy with Proposition \ref{prop:strata of moduli space are orbits of multicurves}, tells us that $\overline{\mathcal{H}}_f$ is the quotient of augmented Teichm\"uller space by the liftable mapping class group $\op{LMod}_f$, and that the irreducible strata of $\overline{\mathcal{H}}_f$ can be interpreted as orbits of multicurves. The result (which we state just for the complex analytic coarse spaces, to avoid orbifold structure technicalities) follows from \cite{hinich2010augmented}.

\begin{proposition}\label{prop:HV10 statement}
There is an isomorphism ${\overline{\mathcal{T}}_{0,B} / \op{LMod}_f \cong \overline{\mathcal{H}}_f}$ of complex analytic spaces. This isomorphism identifies the irreducible strata of $\overline{\mathcal{H}}_f$ with orbits of multicurves under the action of $\op{LMod}_f$. Furthermore, containment of closures of strata descends from the reverse containment of multicurves.
\end{proposition}

\begin{proof}
\begin{sloppypar}
By \cite[Theorem 6.1.1]{hinich2010augmented}, the quotient $\overline{\mathcal{T}}_{0,B} / \op{LMod}_f$ has the structure of a complex orbifold. In Section 8.1.1, they describe a morphism of coarse complex analytic spaces ${\nu_{\rho}: \overline{\mathcal{T}}_{0,B} / \op{LMod}_f \rightarrow \overline{\mathcal{H}}^\text{HM}}$ to the Harris-Mumford space of admissible covers. On the level of points, the morphism takes a nodal complex structure to an admissible cover by pulling back complex structures on each component. This map realises $\overline{\mathcal{T}}_{0,B} / \op{LMod}_f$ as the normalisation of the Harris-Mumford space of admissible covers, and by uniqueness of the normalisation, $\overline{\mathcal{T}}_{0,B} / \op{LMod}_f \cong \overline{\mathcal{H}}_f$ as complex analytic spaces.

We write $\op{LMod}_f(\Delta)$ for the image of the stratum $\Delta$ of $\overline{\mathcal{T}}_{0,B}$ under the quotient map. $\op{LMod}_f(\Delta)$ is connected, as the image of a connected set under a continuous map, and by the description of the morphism on the level of points, $\op{LMod}(\Delta)$ is a connected component of the preimage under $\pi_B$ of the stratum $\mathscr{T}_\Delta$. Thus, $\op{LMod}_f(\Delta)$ is an irreducible stratum of $\overline{\mathcal{T}}_{0,B} / \op{LMod}_f \cong \overline{\mathcal{H}}_f$. 
\end{sloppypar}
\end{proof}
\end{subsection}
\end{section}

\newpage
\begin{section}{Embedded stable trees}\label{section:embedded stable trees}
The purpose of this section is to give the topological set-up of embedded stable trees, dual multicurves and braiding that we will use later on.

\begin{subsection}{Embedded stable trees} Let $B$ be a finite set of points in $S^2$ with $|B| \geq 3$, and let $\mathscr{T}$ be a $B$-marked stable tree. An \textit{embedding} of $\mathscr{T}$ into $(S^2,B)$ maps the endpoint of the leg $b$ to the point $b \in S^2$. We consider embeddings of stable trees up to isotopy in $(S^2,B)$: two embeddings of $\mathscr{T}$ are \textit{isotopic} if they are isotopic (as maps) through embeddings of $\mathscr{T}$.

\begin{remark}\label{remark:ambient isotopy}
If two embeddings $\iota, \iota': \mathscr{T} \hookrightarrow S^2$ are isotopic, then there is actually an ambient isotopy of the sphere $I_t: S^2 \rightarrow S^2$ that fixes $B$ throughout and takes one embedding to the other, in that $I_0 = \op{Id}$ and $I_1 \circ \iota = \iota'$. This follows from the Alexander trick. We will use this observation to `carry along' topological data associated to an embedding (such as a dual multicurve, defined below) to an isotopic embedding. 
\end{remark}
\end{subsection}

\begin{subsection}{Embedded stable trees notation} 
We often refer to an embedding $\iota: \mathscr{T} \hookrightarrow S^2$ by just its image $T$, suppressing the map. A vertex, edge, leg or half-edge of $T$ is the image of a vertex, edge, leg or half-edge, respectively, of $\mathscr{T}$. We often use the same symbol for corresponding features of $\mathscr{T}$ and $T$ when there is no risk of confusion.

Let $v$ be a vertex of $T$. Suppose $h_1, \ldots, h_k$ are the half-edges incident with $v$, appearing in that order going anticlockwise. We record their cyclic order in a cycle $(h_1 \, h_2 \, \ldots \, h_k)$. Putting all such cycles together -- one for each vertex of $T$ -- we obtain a permutation of the set of half-edges $H(T)$ which we denote by $\mathit{ord}_T \in S_{H(T)}$. 
\end{subsection}

\begin{subsection}{Dual multicurves}\label{subsubsec:dual multicurves}
Suppose $T$ is an embedding of a $B$-marked stable tree in $(S^2,B)$, with $e$ an (open) edge. Let $\delta_e$ be a simple closed curve that intersects $T$ exactly once, along $e$, with both components of $S^2 \smallsetminus \delta_e$ containing one of the two resulting subtrees. The curve $\delta_e$ is unique up to isotopy $\op{rel. } B$. The \textit{dual multicurve}\footnote{This is an abuse of notation because we need to make choices of $\delta_e$, but as we will only consider the dual multicurve up to isotopy, these choices don't matter.} to $T$ is 
$$\Delta_T = \left\{\delta_e \mid e \in E(T)\right\}.$$
By induction, these curves can be chosen to be disjoint. They are essential because each component of $S^2 \smallsetminus \delta_e$ contains at least two points of $B$, and they are non-homotopic because (again by induction) each $\delta_e$ partitions $B$ differently. Figure \ref{fig:dual multicurve example} shows an example. If $T$ and $T'$ are isotopic, then $\Delta_T$ and $\Delta_{T'}$ are isotopic -- an ambient isotopy (Remark \ref{remark:ambient isotopy}) carries a choice of $\Delta_T$ to a choice of $\Delta_{T'}$.

\begin{figure}[ht!]
\centering
\includegraphics[scale=1]{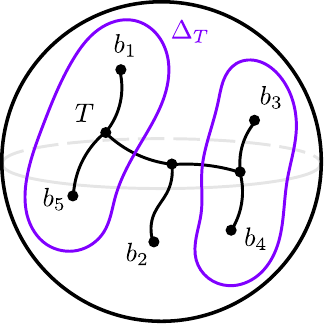}
\caption{An embedding $T$ of a $B$-marked stable tree with its dual multicurve $\Delta_T$.}
\label{fig:dual multicurve example}
\end{figure}

\end{subsection}

\begin{subsection}{Braiding embedded stable trees}\label{subsection:braiding embedded stable trees}
There are operations on the isotopy classes of embeddings of $\mathscr{T}$ that we call \textit{braid moves}. A braid move changes how just one half-edge $h$ of $\mathscr{T}$ is embedded, shifting it one step clockwise or anticlockwise in the cyclic order of the half-edges around $\op{base}(h)$.

To streamline the next definition, we define some special loops: if $h$ is a half-edge of an embedding $T$, then we let $\gamma_h$ be a simple loop based at $\op{base}(h)$ that bounds (on its left) the subtree of $T \smallsetminus \op{base}(h)$ containing $h$, with $\gamma_h$ intersecting $T$ only at $\op{base}(h)$. The loop $\gamma_h$ is unique up to isotopy $\op{rel. } B$. Note that if $h$ is a leg, then $\gamma_h$ is a keyhole loop. These loops will reappear in Section \ref{section:decorations from embedded stable trees}.

\begin{definition}
Let $T$ be an embedding of a $B$-marked stable tree $\mathscr{T}$ in $(S^2,B)$, and let $h$ be a half-edge of $T$. Suppose $\widetilde{h} = \mathit{ord}_T(h)$ is the next half-edge anticlockwise around $\op{base}(h)$. Following $\gamma_{\widetilde{h}}^{-1}$ and then $h$ gives a path whose interior intersects $T$ once, at $\op{base}(h)$. Nudge this path by an isotopy to remove this intersection, and call the resulting curve $k$. Let $T'$ be the embedding of $\mathscr{T}$ obtained by replacing $h$ with $k$. We say that $T'$ differs from $T$ by an \textit{(anticlockwise) braid move} at $h$.
\end{definition}

An anticlockwise braid move at $h$ changes the cyclic ordering of the half-edges around $\op{base}(h)$: the half-edge $h$ gets shifted one step anticlockwise. This translates to a conjugation of $\mathit{ord}_T$ by a transposition:
$$\mathit{ord}_{T'} = (h\,\, \widetilde{h}) \circ \mathit{ord}_T \circ (h \,\, \widetilde{h})^{-1},$$
where $\widetilde{h} = \mathit{ord}(h)$ is the next half-edge anticlockwise. We define \textit{clockwise} braid moves similarly, and we refer to a clockwise or anticlockwise braid move simply as a \textit{braid move}.

Braid moves are well-defined operations on the isotopy classes of embeddings of a stable tree. Some examples of anticlockwise braid moves are shown in Figure \ref{fig:braid moves}. \textit{Braiding} an embedding means applying a sequence of braid moves. The next proposition says that braiding an embedding does not change the dual multicurve.

\begin{figure}[hbtp]
\centering
\includegraphics[scale=1]{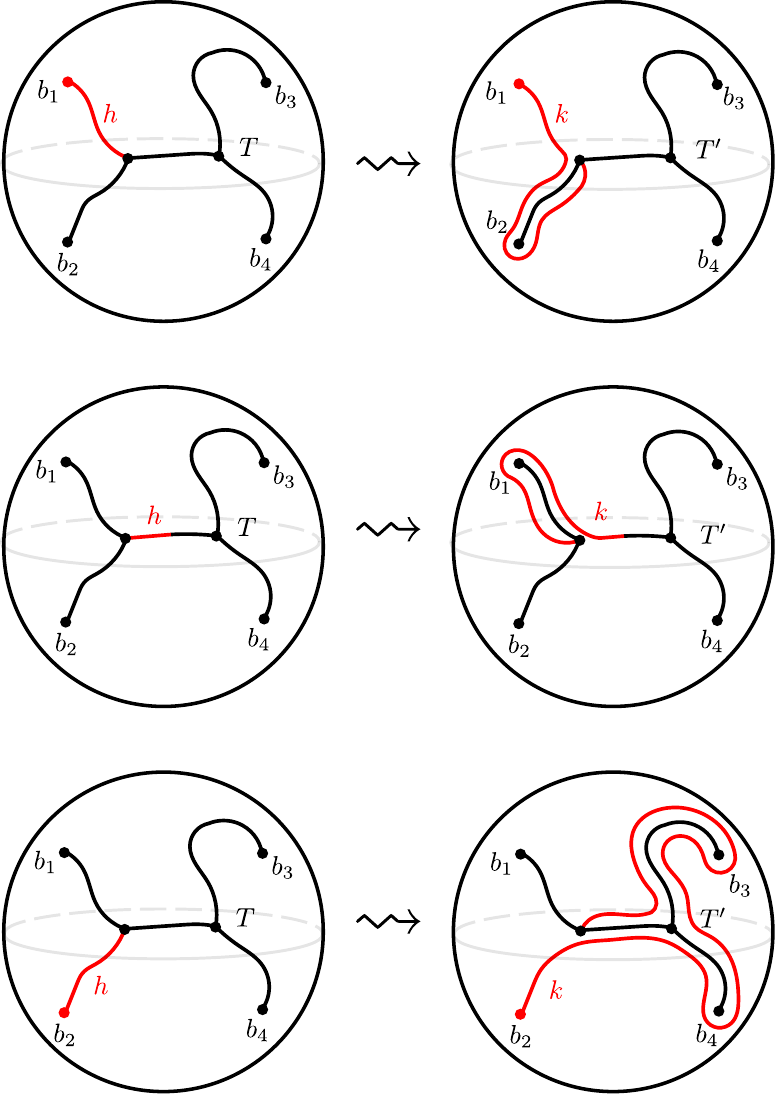}
\caption{Examples of anticlockwise braid moves on an embedding $T$ at the half-edges shown to produce new embeddings.}
\label{fig:braid moves}
\end{figure}

\begin{proposition}
Let $T$ be an embedding of a $B$-marked stable tree in $(S^2,B)$. Suppose $T'$ differs from $T$ by a braid move. Then $\Delta_T$ and $\Delta_{T'}$ are isotopic. 
\end{proposition}

\begin{proof}
The loop $\gamma_h$ in the definition of a braid move can be chosen disjoint from $\Delta_T$, so that $\Delta_T$ is a common choice of dual multicurve for both $T$ and $T'$. 
\end{proof}

Conversely, if two embeddings $T$ and $T'$ both have the same dual multicurve, then one can be braided to give the other.

\begin{proposition}\label{prop:dual trees connected by simple moves}
Suppose $T$ and $T'$ are embeddings of a $B$-marked stable tree $\mathscr{T}$ in $(S^2,B)$ such that $\Delta_T$ and $\Delta_{T'}$ are isotopic. Then there exists a sequence of embeddings of $\mathscr{T}$
$$T=T_0,\,\, T_1,\,\, \ldots,\,\, T_N$$
such that $T_i$ differs from $T_{i-1}$ by a braid move for $i=1, \ldots, N$, and $T_N$ is isotopic to $T'$.
\end{proposition}

This proof is longer and more technical, so we relegate it to Appendix \ref{sec:appendix}.
\end{subsection}

\begin{subsection}{Contraction of embedded stable trees}
Suppose $T$ is an embedding and $e$ is an edge that is incident with a vertex $v$. A \textit{contraction along $e$ fixing $v$} is an embedding $T'$ obtained from $T$ by contracting in the edge $e$ towards $v$, keeping $v$ fixed, as shown in Figure \ref{fig:contraction of embedded stable trees}. A sequence of such contractions is called a \textit{contraction} of $T$.

\begin{figure}[ht!]
\centering
\includegraphics[scale=1]{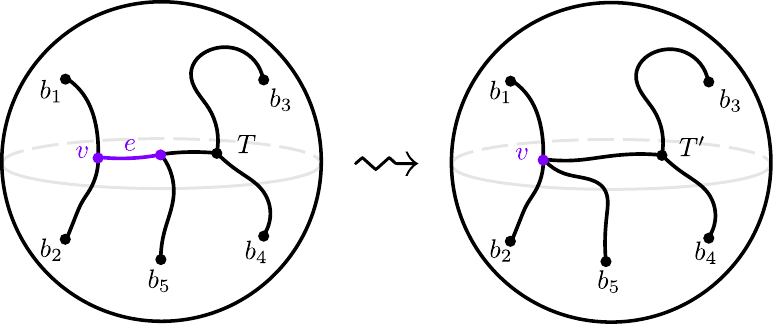}
\caption{Contracting $T$ along the edge $e$, keeping $v$ fixed, produces the embedding $T'$.}
\label{fig:contraction of embedded stable trees}
\end{figure}

\end{subsection}
\end{section}

\newpage
\begin{section}{Decorated trees}\label{sec:decorated trees}
\begin{sloppypar}
In this section, we define decorated trees. These are the combinatorial objects that we will use to describe the strata of compactified Hurwitz spaces. Throughout, let ${P = (g, d, B, A, \phi : A \rightarrow B, \op{br}, \op{rm})}$ be a portrait.
\end{sloppypar}

We write $H(\mathscr{T})$ for the set of half-edges of a $B$-marked stable tree $\mathscr{T}$, and we often simply write $b$ for the leg labelled $b\in B$.

\begin{subsection}{Decorations} 
\begin{sloppypar}
We should think of a decoration of $\mathscr{T}$ as modelling certain monodromy information that comes from embedding $\mathscr{T}$ into the target sphere of a branched cover ${f:(X,A) \rightarrow (S^2,B)}$. In particular, an embedding of $\mathscr{T}$ gives a cyclic ordering of the half-edges around each vertex, and picks out certain loops (the $\gamma_h$ from Section \ref{subsection:braiding embedded stable trees}) whose monodromy permutations we'd like to keep track of. We'll see this encoding in action in Section \ref{subsec:the decoration construction}.
\end{sloppypar}

\begin{definition}\label{def:decoration} 
Let $\mathscr{T}$ be a $B$-marked stable tree. A \textit{$P$-decoration} of $\mathscr{T}$ is given by the data $(\mathit{ord}, \mathit{mon}, \mathit{cyc})$ where
\begin{enumerate}
\item $\mathit{ord}$ is a permutation of the set $H(\mathscr{T})$ of half-edges, with one disjoint cycle for each vertex $v$ that contains all the half-edges incident with $v$
\item $\mathit{mon}$ is an assignment
$$\mathit{mon}:H(\mathscr{T}) \longrightarrow S_d$$
of a permutation of $S_d$ to each half-edge of $\mathscr{T}$
\item $\mathit{cyc}$ is an injection
$$\mathit{cyc}: A \longrightarrow \bigsqcup_{b\in B} \op{Cycles}\left(\mathit{mon}(b)\right)$$
from $A$ to the disjoint union of cycles of the leg permutations
\end{enumerate}
satisfying the conditions
\begin{enumerate}[label={(\roman*.)}]
\item for each cycle $(h_1 \, h_2 \, \ldots \, h_k)$ of $\mathit{ord}$,
$$\mathit{mon}(h_k)\ldots \mathit{mon}(h_2)\mathit{mon}(h_1) = \op{Id} \in S_d$$
\item for each edge, the permutations on its two half-edges are inverses
\item for each $b \in B$, the leg permutation $\mathit{mon}(b)$ has cycle type $\op{br}(b)$
\item for each $a \in A$, $\mathit{cyc}(a)$ is a cycle of $\mathit{mon}\left({\phi(a)}\right)$ of length $\op{rm}(a)$
\item the collection of leg permutations $\left\{\mathit{mon}(b)\right\}_{b \in B}$ is transitive
\end{enumerate}
A $B$-marked stable tree $\mathscr{T}$ together with a $P$-decoration is called a \textit{$P$-decorated tree}, or simply a \textit{decorated tree} if the portrait is understood. The set of $P$-decorated trees is denoted by $\op{Dec}_{0,B}(P)$.
\end{definition}

\begin{remark}\label{remark:decoration on the legs defines the rest} ~
\begin{enumerate}[label=(\arabic*)]
\item The purpose of condition (i.) is to ensure that the (anticlockwise) cyclic product of the permutations around each vertex is the identity.
\item A decoration $(\mathit{ord}, \mathit{mon}, \mathit{cyc})$ is determined by $\mathit{ord}$, $\mathit{mon}|_B$ and $\mathit{cyc}$, because once we know the permutations on the legs, we can `propagate in' using conditions (i) and (ii) to determine the permutations on the edges. 
\end{enumerate}
\end{remark}

\begin{example} Let $P$ be the degree-$3$ portrait from Section \ref{subsec:portraits}:
\[
\begin{tikzcd}[column sep=small, row sep=12mm, column sep=2.7mm, ampersand replacement=\&]
\overset{\displaystyle a_2}{\critpt}\arrow[dr,"\scriptstyle 2" left, line width=0.5mm] 
\& \&
\stdpt\arrow[dl,"\scriptstyle 1" right, line width=0.5mm]
\& \& ~\&
\overset{\displaystyle a_3}{\critpt}\arrow[dr, "\scriptstyle 2" left, line width=0.5mm] 
\& \&
\stdpt\arrow[dl,"\scriptstyle 1" right, line width=0.5mm]
\& \& ~\&
\critpt\arrow[dr,"\scriptstyle 2" left, line width=0.5mm]
\& \&
\overset{\displaystyle a_4}{\stdpt}\arrow[dl,"\scriptstyle 1" right, line width=0.5mm]
\& \& ~\&
\critpt\arrow[dr,"\scriptstyle 2" left, line width=0.5mm]
\& \&
\overset{\displaystyle a_1}{\stdpt}\arrow[dl,"\scriptstyle 1" right, line width=0.5mm]
\\
\&
\underset{\displaystyle b_1}{\stdpt} 
\& \& \& \& \&
\underset{\displaystyle b_2}{\stdpt} 
\& \& \& \& \&
\underset{\displaystyle b_3}{\stdpt} 
\& \& \& \& \&
\underset{\displaystyle b_4}{\stdpt} 
\&
\end{tikzcd}
\]
Decorated trees are conveniently presented by diagrams such as Figure \ref{fig:decoration first example}. Letting $h$ and $h'$ be the half-edges (left-to-right) of the central edge,
\begin{enumerate}
\item $\mathit{ord} = (h\, b_1 \, b_2)(h' \, b_3 \, b_4)$
\item ${\mathit{mon}: b_1, b_2, b_3, b_4 \mapsto (1\,2)(3),\, (2\,3)(1),\, (1\,2)(3),\, (1\,3)(2)}$ and 
\newline $\mathit{mon}: h, h' \mapsto (1\,2\,3),\, (1\,3\,2).$
\item $\mathit{cyc}: a_1, a_2, a_3, a_4 \mapsto (2), (1\,2), (2\,3), (3)$ in $\mathit{mon}(b_4), \mathit{mon}(b_1), \mathit{mon}(b_2), \mathit{mon}(b_3)$ respectively
\end{enumerate} 
Conditions (i.)-(v.) are satisfied.

\begin{figure}[ht!]
\centering
\includegraphics[scale=1]{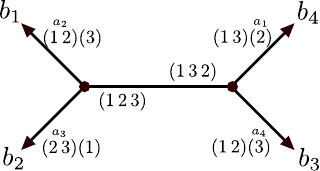}
\caption{A $P$-decoration of a stable tree.}
\label{fig:decoration first example}
\end{figure}

\end{example}
\end{subsection}

\begin{subsection}{Contraction of decorations} 
\begin{sloppypar} 
Contracting an edge (with half-edges $h$, $h'$, say) of a $B$-marked stable tree $\mathscr{T}$ yields another $B$-marked stable tree, $\mathscr{T}'$. A $P$-decoration ${(\mathit{ord}, \mathit{mon}, \mathit{cyc})}$ of $\mathscr{T}$ induces a $P$-decoration ${(\mathit{ord}', \mathit{mon}', \mathit{cyc}')}$ of $\mathscr{T}'$ in the obvious way: ${\mathit{mon}'(k) = \mathit{mon}(k)}$ for all ${k \in H(\mathscr{T}')}$, ${\mathit{cyc}'(a) = \mathit{cyc}(a)}$ for all ${a \in A}$, and if 
$$\mathit{ord} = \ldots(h\,\, h_1\, \ldots \, h_k)(h' \,\, h'_1 \, \ldots \, h'_l)\ldots$$
(written with $h$ and $h'$ starting their respective cycles) then 
$$\mathit{ord}' = \ldots(h_1\, \ldots \, h_k \,\, h'_1 \, \ldots \, h'_l)\ldots.$$
This is shown for our example in Figure \ref{fig:contracting decoration first example}. Conditions (i.)-(v.) are satisifed by ${(\mathit{ord}', \mathit{mon}', \mathit{cyc}')}$. We call the new decorated tree an \textit{edge contraction} of the original decorated tree. More generally, a decorated tree obtained by a sequence of edge contractions is called a \textit{contraction} of the original.
\end{sloppypar}

\begin{figure}[ht!]
\centering
\includegraphics[scale=1]{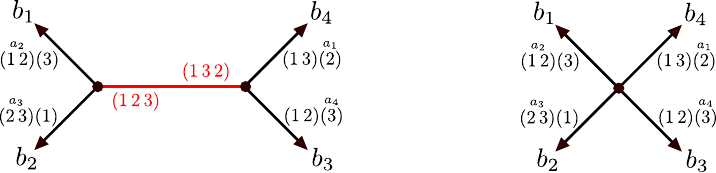}
\caption{Contracting the red edge induces a decoration of the resulting stable tree.}
\label{fig:contracting decoration first example}
\end{figure}
\end{subsection}

\begin{subsection}{Hurwitz equivalence of decorations}
\begin{sloppypar}
There are two types of operations on \mbox{$P$-decorated} trees that determine the correct notion of equivalence. The first are \textit{global conjugations}, which capture the idea that relabelling the permutations of a $P$-decorated tree by some $\tau \in S_d$ should yield an equivalent $P$-decorated tree. The second are \textit{braid moves (of decorated trees)}, which are more subtle and will be justified when we examine the interplay between the combinatorics and the topology of branched covers. We say that two $P$-decorated trees are \textit{Hurwitz equivalent} if one can be transformed to the other by a sequence of global conjugations and braid moves. 
\end{sloppypar}

\begin{definition} Let $(\mathit{ord}, \mathit{mon}, \mathit{cyc})$ be a $P$-decoration of $\mathscr{T}$, and let $\tau \in S_d$. Define a new decoration $(\mathit{ord}', \mathit{mon}', \mathit{cyc}')$ of $\mathscr{T}$ as follows.
\begin{enumerate}
\item $\mathit{ord}' = \mathit{ord}$
\item for all $h \in H(\mathscr{T})$, 
$$\mathit{mon}'(h) = \tau \circ \mathit{mon}(h) \circ \tau^{-1}$$
\item for all $a \in A$, 
$$\mathit{cyc}'(a) = \tau \circ \mathit{cyc}(a) \circ \tau^{-1}$$
\end{enumerate}
We say that $(\mathit{ord}', \mathit{mon}', \mathit{cyc}')$ is obtained from $(\mathit{ord}, \mathit{mon}, \mathit{cyc})$ by a \textit{global conjugation} by $\tau$, and that these decorations are \textit{globally conjugate}.
\end{definition}

The new decoration satisfies conditions (i.)-(v.) and is thus a $P$-decoration of $\mathscr{T}$. A global conjugation is shown in Figure \ref{fig:globally conjugate decorations example}, where $P$ is the portrait in our running example and $\tau = (1\,2\,3)$.

\begin{figure}[ht!]
\centering
\includegraphics[scale=1]{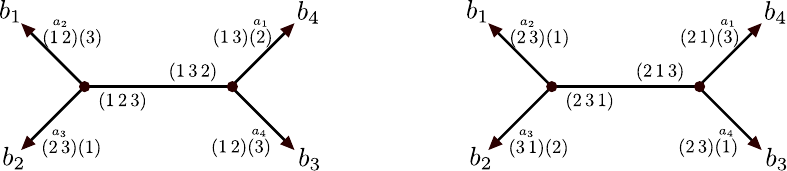}
\caption{Two decorations that are globally conjugate by $\tau = (1\,2\,3)$.}
\label{fig:globally conjugate decorations example}
\end{figure}

The second type of operations -- braid moves of decorated trees -- allow us to accommodate decorations that have different cyclic orders. Braid moves of decorated trees encode the effect that braid moves of \textit{embedded} trees have on certain monodromy permutations. We will make this precise in Section \ref{subsec:effect of braiding T on decoration}.

\begin{definition}
Let $(\mathit{ord}, \mathit{mon}, \mathit{cyc})$ be a $P$-decoration of $\mathscr{T}$, and let $h$ be a half-edge. Let $\widetilde{h} = \mathit{ord}(h)$ be the next half-edge anticlockwise. Define a new decoration $(\mathit{ord}', \mathit{mon}', \mathit{cyc}')$ of $\mathscr{T}$ as follows.
\begin{enumerate}
\item $\mathit{ord}' = (h \,\, \widetilde{h}) \circ \mathit{ord} \circ (h \,\, \widetilde{h})^{-1}$
\item for all half-edges $k$ in the component of $\mathscr{T} \smallsetminus \{\op{base}(h)\}$ containing $h$, 
$$\mathit{mon}'(k) = \mathit{mon}(\widetilde{h}) \circ \mathit{mon}(k) \circ \mathit{mon}(\widetilde{h})^{-1}$$
and for all other half-edges $k$, 
$$\mathit{mon}'(k) = \mathit{mon}(k)$$
\item for all $a \in A$ so that ${\phi(a)}$ is a leg in the component of $\mathscr{T} \smallsetminus \{\op{base}(h)\}$ containing $h$, 
$$\mathit{cyc}'(a) = \mathit{mon}(\widetilde{h}) \circ \mathit{cyc}(a) \circ \mathit{mon}(\widetilde{h})^{-1},$$
and for all other $a \in A$, 
$$\mathit{cyc}'(a) = \mathit{cyc}(a)$$
\end{enumerate}
We say that $(\mathit{ord}', \mathit{mon}', \mathit{cyc}')$ is obtained from $(\mathit{ord}, \mathit{mon}, \mathit{cyc})$ by an \textit{(anticlockwise) braid move} of decorated trees at the half-edge $h$.
\end{definition}

The new data satisfies conditions (i.)-(v.) and is thus a $P$-decoration. In words, an anticlockwise braid move shifts $h$ one step anticlockwise through $\widetilde{h} = \mathit{ord}(h)$, and the price of this reordering is that all the permutations on the shifted subtree get conjugated by $\mathit{mon}(\widetilde{h})$.

We define clockwise braid moves of decorated trees similarly: a clockwise braid move shifts $h$ one step clockwise through $\widetilde{h} = \mathit{ord}^{-1}(h)$, and all the permutations on the shifted subtree pick up a conjugation by $\mathit{mon}(\widetilde{h})^{-1}$. A clockwise or anticlockwise braid move is simply a \textit{braid move}.

A braid move for our running-example portrait is shown in Figure \ref{fig:simple reordering first example}, and a bigger example (for a slightly modified portrait with $|B|=5$) is shown in Figure \ref{fig:simple reordering second example}.

\begin{figure}[ht!]
\centering
\includegraphics[scale=1]{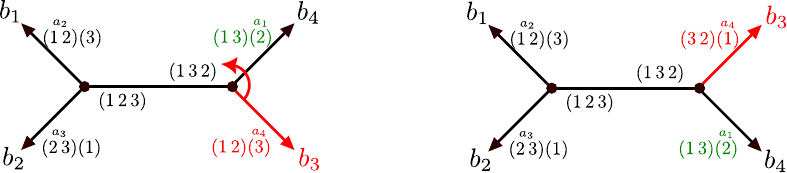}
\caption{An anticlockwise braid move at $h=b_3$ through $\widetilde{h}=b_4$: $h$ gets shifted anticlockwise and the red permutations and $A$-markings get conjugated by $\mathit{mon}(\widetilde{h}) = (1\,3)(2)$.}
\label{fig:simple reordering first example}
\end{figure}

\begin{figure}[ht!]
\centering
\includegraphics[scale=1]{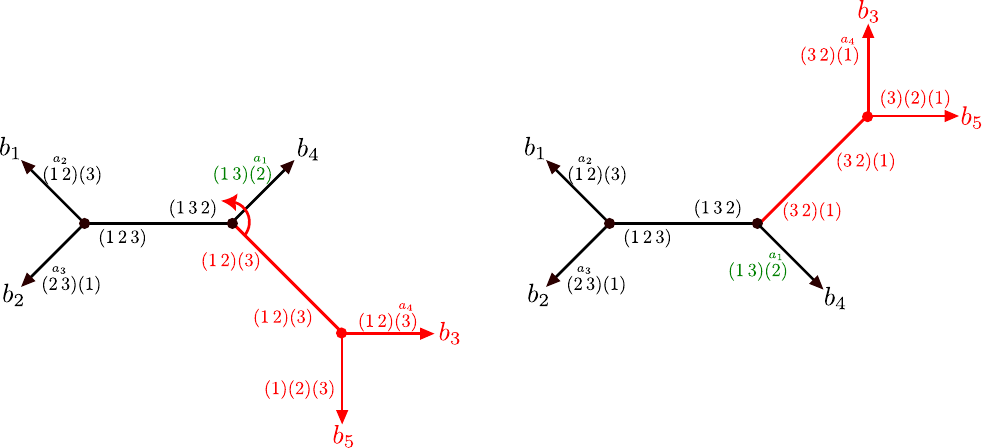}
\caption{A braid move at $h$ (the half-edge by the red anticlockwise arrow) through $\widetilde{h} = b_4$: $h$ gets shifted anticlockwise, and the red permutations and $A$-markings get conjugated by $\mathit{mon}(\widetilde{h}) = (1\,3)(2)$.}
\label{fig:simple reordering second example}
\end{figure}

\newpage
Combining global conjugations and braid moves gives us the right equivalence relation, which we call \textit{Hurwitz equivalence}.

\begin{definition}\label{def:hurwitz equivalence of decorated trees}
Two $P$-decorations of a stable tree $\mathscr{T}$ are \textit{Hurwitz equivalent} if one can be obtained from the other by a sequence of global conjugations and braid moves of decorated trees. We write 
$$\op{Stab}_{0,B}(P) = \op{Dec}_{0,B}(P) / \text{Hurwitz equivalence}$$
for the set of Hurwitz equivalence classes of $P$-decorated trees. 
\end{definition}

The next lemma tells us that contraction of $P$-decorated trees descends to a well-defined notion of contraction on Hurwitz equivalence classes of $P$-decorated trees.

\begin{lemma}\label{lemma:contraction of Heq classes is well defined}
\begin{sloppypar}
Let $\mathscr{T}$ be a stable tree that contracts along an edge $e$ to $\mathscr{T}'$. Suppose ${(\mathit{ord}_1, \mathit{mon}_1, \mathit{cyc}_1)}$ and $(\mathit{ord}_2, \mathit{mon}_2, \mathit{cyc}_2)$ are Hurwitz equivalent $P$-decorations of $\mathscr{T}$ that contract along $e$ to give the $P$-decorations $(\mathit{ord}_1', \mathit{mon}_1', \mathit{cyc}_1')$ and $(\mathit{ord}_2', \mathit{mon}_2', \mathit{cyc}_2')$ of $\mathscr{T}'$, respectively. Then $(\mathit{ord}_1', \mathit{mon}_1', \mathit{cyc}_1')$ and $(\mathit{ord}_2', \mathit{mon}_2', \mathit{cyc}_2')$ are Hurwitz equivalent. 
\end{sloppypar}
\end{lemma}

\begin{proof}
\begin{sloppypar}
If $(\mathit{ord}_1, \mathit{mon}_1, \mathit{cyc}_1)$ and $(\mathit{ord}_2, \mathit{mon}_2, \mathit{cyc}_2)$ differ by a global conjugation by $\tau$, then $(\mathit{ord}_1', \mathit{mon}_1', \mathit{cyc}_1')$ and $(\mathit{ord}_2', \mathit{mon}_2', \mathit{cyc}_2')$ also differ by a global conjugation by $\tau$.

If $(\mathit{ord}_1, \mathit{mon}_1, \mathit{cyc}_1)$ and $(\mathit{ord}_2, \mathit{mon}_2, \mathit{cyc}_2)$ differ by a braid move at the half-edge $h$ (without loss of generality an \textit{anticlockwise} braid move), then the situation is slightly more complicated. If neither $h$ nor $\mathit{ord}_1(h)$ is in $e$, then $(\mathit{ord}_1', \mathit{mon}_1', \mathit{cyc}_1')$ and $(\mathit{ord}_2', \mathit{mon}_2', \mathit{cyc}_2')$ differ by a braid move at $h$. If $h$ or $\mathit{ord}_1(h)$ is in $e$, then $(\mathit{ord}_1', \mathit{mon}_1', \mathit{cyc}_1')$ and $(\mathit{ord}_2', \mathit{mon}_2', \mathit{cyc}_2')$ differ by a sequence of braid moves (involving each of the new half-edges at $\op{base}(h)$ created by contracting $e$).
\end{sloppypar}
\end{proof}
\end{subsection}
\end{section}

\newpage
\begin{section}{Decorated trees from embedded stable trees}\label{section:decorations from embedded stable trees}
In this section, we explain how to construct a decorated tree given a branched cover and an embedded tree. This construction requires us to make some choices, and the second half of the section will examine how these choices affect the decorated tree.

\begin{subsection}{The decoration construction}\label{subsec:the decoration construction}
Recall the definition of the loop $\gamma_h$ from Section \ref{section:embedded stable trees}: if $h$ is a half-edge of an embedding $T$, then $\gamma_h$ is a simple loop (unique up to isotopy $\op{rel. } B$) based at $\op{base}(h)$ that bounds the subtree of $T \smallsetminus \op{base}(h)$ containing $h$ on its left, with $\gamma_h$ intersecting $T$ only at $\op{base}(h)$. If $h$ is a leg, then $\gamma_h$ is a keyhole loop.

Recall also our convention that concatenation of path is written right-to-left, to match function composition.

\begin{definition}\label{cons:decoration of stable tree}
\begin{sloppypar}
Let ${f:(X,A) \rightarrow (S^2,B)}$ be a branched cover with $X$ connected, and let $\mathscr{T}$ be a $B$-marked stable tree. Choose an embedding $T$ of $\mathscr{T}$, a vertex $v$ of $T$, and a labelling ${l: f^{-1}(v) \rightarrow \{v_1, \ldots, v_d\}}$. Let ${\Sigma: \pi_1(S^2 \smallsetminus B, v) \rightarrow S_d}$ be the monodromy representation induced by the labelling. With these choices, we construct a $P(f)$-decoration 
$$D_f(T, v, l) = (\mathit{ord}, \mathit{mon}, \mathit{cyc})$$
of $\mathscr{T}$, where:
\begin{enumerate}
\item $\mathit{ord} = \mathit{ord}_T$
\item for $h \in H(\mathscr{T})$, let $\alpha_h$ be the unique shortest path in $T$ from $v$ to $\op{base}(h)$ and let ${\overline{\gamma}_h = \alpha_h^{-1} \concat \gamma_h \concat \alpha_h}$; set 
$$\mathit{mon}(h) = \Sigma(\overline{\gamma}_h)$$	
\item for $a \in A$, set
$$\mathit{cyc}(a) = \mathit{cyc}_{\overline{\gamma}_b}(a),$$
where $b = f(a)$ and $\mathit{cyc}_{\overline{\gamma}_b}$ is the canonical bijection of Proposition \ref{prop:keyhole loop monodromy}
\end{enumerate}
\end{sloppypar}
\end{definition}

An illustration of the construction is given in Figure \ref{fig:decoration from embedded tree}, with $P(f)$ our running-example portrait from Section \ref{subsec:portraits}.

\begin{lemma}
The construction in Definition \ref{cons:decoration of stable tree} yields a $P(f)$-decoration of $\mathscr{T}$. 
\end{lemma}

\begin{proof}
We must show that conditions (i.)-(v.) of Definition \ref{def:decoration} are satisfied.
\begin{enumerate}[label={(\roman*.)}]
\item Let $w$ be a vertex of $\mathscr{T}$, and let $h_1, \ldots, h_k$ be its incident half-edges in anticlockwise cyclic order. The concatenation $\overline{\gamma}_{h_k} \concat \ldots \concat \overline{\gamma}_{h_1}$ is homotopic to the boundary of a disk containing the tree $T$, which can be homotoped to a point, so
\begin{align*}
\mathit{mon}(h_k) \circ \ldots \circ \mathit{mon}(h_1) &= \Sigma\left(\overline{\gamma}_{h_k}\right) \circ \ldots \circ \Sigma\left(\overline{\gamma}_{h_1}\right)\\ 
&= \op{Id}.
\end{align*}
\item Suppose the half-edges $h$ and $h'$ form an edge. The loops $\overline{\gamma}_h$ and $\overline{\gamma}_{h'}^{-1}$ are homotopic in $(S^2,B)$, so $\Sigma(\overline{\gamma}_h) = \Sigma(\overline{\gamma}_{h'})^{-1}$ and $\mathit{mon}(h) = \mathit{mon}(h')^{-1}$.
\item $\overline{\gamma}_{b}$ is a keyhole loop, so by Proposition \ref{prop:keyhole loop monodromy}, $\mathit{mon}(b) = \Sigma(\overline{\gamma}_b)$ has cycle type $\op{br}_f(b)$. 
\item $\overline{\gamma}_{f(a)}$ is a keyhole loop, so by Proposition \ref{prop:keyhole loop monodromy}, $\mathit{cyc}(a)$ is a cycle of $\mathit{mon}({f(a)}) = \Sigma(\overline{\gamma}_{f(a)})$ of length $\op{rm}_f(a)$. 
\item $X$ is connected, so by Proposition \ref{prop:transitive monodromy group implies connected source}, the collection of leg permutations $\{\mathit{mon}(b)\}_{b \in B}$ is transitive.
\end{enumerate}
\end{proof}

\begin{figure}[ht!]
\centering
\includegraphics[scale=1]{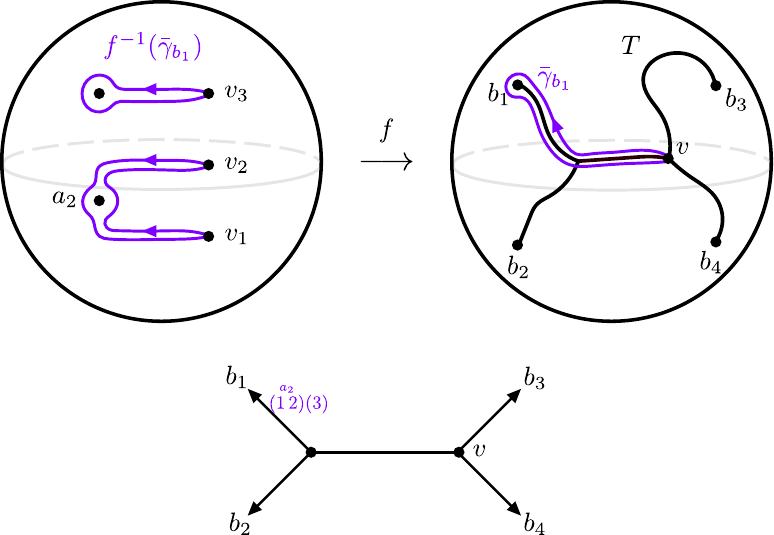}
\caption{Given a branched cover $f$, an embedding $T$, a vertex $v$ and a labelling of $f^{-1}(v)$, we use the loop $\overline{\gamma}_{b_1}$ to decorate the half-edge $b_1$.}
\label{fig:decoration from embedded tree}
\end{figure}

The rest of this section will be spent proving properties of the decoration function. In particular, we will examine how $D_f(T,v, l)$ changes when we make different choices of $f$, $T$, $v$ and $l$. 
\end{subsection}

\begin{subsection}{Isotoping \texorpdfstring{\boldmath $T, v$ and $l$}{T, v and l}} 
\begin{sloppypar}
Suppose $\beta$ is a path in $S^2 \smallsetminus B$ from $v$ to $v'$, and suppose ${l: f^{-1}(v) \rightarrow \{v_1, \ldots, v_d\}}$ is a labelling. The \textit{extension of $l$ along $\beta$} is the labelling ${l': f^{-1}(v') \rightarrow \{v_1', \ldots, v_d'\}}$ defined so that the lift of $\beta$ to $v_i$ ends at $v_i'$. 
\end{sloppypar}

\begin{lemma}\label{lemma:isotoping T,v,l}
Let $f, T, v$ and $l$ be as above, and suppose there is an isotopy from $T$ to $T'$ that takes $v$ to $v'$. Let $l'$ be the extension of $l$ along the path that $v$ takes to $v'$ under the isotopy. Then 
$$D_f(T,v,l) = D_f(T', v', l').$$
\end{lemma}

\begin{proof}
Let the two decorations be $(\mathit{ord}, \mathit{mon}, \mathit{cyc})$ and $(\mathit{ord}', \mathit{mon}', \mathit{cyc}')$, respectively. The isotopy preserves the cyclic order of the half-edges around each vertex, so $\mathit{ord} = \mathit{ord}'$. Let $h$ be a half-edge of $T$, and let $h'$ be the half-edge of $T'$ that $h$ gets isotoped to. An ambient isotopy of the sphere (Remark \ref{remark:ambient isotopy}) that takes $T$ to $T'$ carries the loop $\gamma_h$ to a valid choice of loop $\gamma_{h'}$, and thus carries $\overline{\gamma}_h$ (based at $v$) to a valid choice of $\overline{\gamma}_{h'}$ (based at $v'$). The fact that $l'$ is an extension of $l$ along the path from $v$ to $v'$ ensures exactly that the monodromy permutation of $\overline{\gamma}_h$ is equal to the monodromy permutation of $\overline{\gamma}_{h'}$ (and that the cycle markings are equal, in the case that $h$ and $h'$ are legs), and so $\mathit{mon} = \mathit{mon}'$ and $\mathit{cyc} = \mathit{cyc}'$. 
\end{proof}
\end{subsection}

\begin{subsection}{Changing \texorpdfstring{\boldmath $v$ and $l$}{v and l}}
Next comes the easiest: changing $v$ and $l$ while keeping $T$ fixed changes the decoration by a global conjugation.

\begin{lemma}\label{lemma:changing vertex and labelling of decoration}
With the notation above, suppose that $v'$ is another choice of vertex (possibly equal to $v$) and that $l': f^{-1}(v') \rightarrow \{v'_1, \ldots, v'_d\}$ is a labelling of its preimage. Then the decorations $D_f(T,v,l)$ and $D_f(T,v',l')$ are globally conjugate. 
\end{lemma}

\begin{proof}
Let the two decorations be $(\mathit{ord}, \mathit{mon}, \mathit{cyc})$ and $(\mathit{ord}', \mathit{mon}', \mathit{cyc}')$, respectively. It is clear that $\mathit{ord} = \mathit{ord}_T = \mathit{ord}'$. Let $\beta$ be the unique shortest path in $T$ from $v$ to $v'$, and let $l''$ be the labelling of $f^{-1}(v')$ obtained by extending the labelling $l$ along the path $\beta$. Then $D_f(T,v,l) = D_f(T,v', l'')$. Composing $l''$ with the correct permutation $\tau$ gives $l'$. The decoration $D_f(T,v',l')$ is then the global conjugation of $D_f(T,v',l'') = D_f(T,v,l)$ by $\tau$.
\end{proof}
\end{subsection}

\begin{subsection}{Braiding \texorpdfstring{\boldmath $T$}{T}}\label{subsec:effect of braiding T on decoration}
As indicated in Section \ref{sec:decorated trees}, changing $T$ by a braid move of embedded trees, while keeping $v$ and $l$ fixed, changes $D_f(T,v,l)$ by a braid move of \textit{decorated} trees.

\begin{lemma}\label{lemma:braid move of trees gives braid move of decorations}
Let $h$ be a half-edge of $T$, and suppose $T'$ is obtained from $T$ by a braid move of embedded trees at the half-edge $h$. Then $D_f(T,\op{base}(h),l)$ differs from $D_f(T',\op{base}(h),l)$ by a braid move of \textit{decorated} trees at $h$.
\end{lemma}

\begin{proof}
\begin{sloppypar}
Suppose the braid move of embedded trees is an anticlockwise braid move, without loss of generality. Let the two decorations be $(\mathit{ord}, \mathord{mon}, \mathit{cyc})$ and $(\mathit{ord}', \mathord{mon}', \mathit{cyc}')$, respectively, and let $h' = \mathit{ord}_T(h)$ be the next half-edge anticlockwise. The cyclic ordering changes as required: ${\mathit{ord}' = \mathit{ord}_{T'} = (h\,\,h') \circ \mathit{ord} \circ (h\,\,h')}$.

Let $k$ be a half-edge in the component of $T \smallsetminus h$ containing $h$. If $\overline{\gamma}_k$ is the original loop that we use to decorate $k$ in $D_f(T, \op{base}(h), l)$, then $\gamma_{h'} \concat \overline{\gamma}_k \concat \gamma_{h'}^{-1}$ is the loop we use to decorate $k$ in $D_f(T', \op{base}(h), l)$. Thus, ${\mathit{mon}'(k) = \mathit{mon}(h') \circ \mathit{mon}(k) \circ \mathit{mon}(h')^{-1}}$. Similarly, if $k$ is a leg and is mapped to by $a \in A$, then ${\mathit{cyc}'(a) = \mathit{mon}(h') \circ \mathit{cyc}(a) \circ \mathit{mon}(h')^{-1}}$.

If $k$ is a half-edge in a different component of $T \smallsetminus h$, then we use the same loop $\overline{\gamma_k}$ to decorate $k$. 
\end{sloppypar}
\end{proof}
\end{subsection}

\begin{subsection}{Realising \texorpdfstring{\boldmath $P$}{P}-decorations} 
Let $X$ be a compact surface of genus $g$, and let $A \subset X$ and $B \subset S^2$ be finite subsets. Let $P = (g, d, B, A, \phi: A \rightarrow B, \op{br}, \op{rm})$ be a portrait. The following proposition tells us that every $P$-decorated tree is actually realised as $D_f(T,v,l)$ for some choice of $f, T, v, l$. The proof is essentially the same as the `Ikea approach' of \cite[Theorem 7.2.2]{cavalieri2016}, restricted to the case where the target is a sphere. In our language, their proof covers the case of the one-vertex stable tree, whereas we make the slight generalisation to other stable trees.

\begin{sloppypar}
\begin{proposition}\label{prop:realising decorations}
Let $(\mathit{ord}, \mathit{mon}, \mathit{cyc})$ be a $P$-decoration of $\mathscr{T}$. Then there exists a branched cover $f:(X,A) \rightarrow (S^2,B)$ realising $P$, an embedding $T$ of $\mathscr{T}$, a vertex $v$ of $T$ and a labelling $l$ of $f^{-1}(v)$ such that ${D_f(T,v,l) = (\mathit{ord}, \mathit{mon}, \mathit{cyc})}$. 
\end{proposition}
\end{sloppypar}

\begin{proof} 
Let $(\mathit{ord}', \mathit{mon}', \mathit{cyc}')$ be the $P$-decoration obtained by contracting all the edges of $\mathscr{T}$ -- the contracted tree has just one vertex, so $\mathit{ord}'$ has exactly one cycle and it contains all the elements of $B$. Let $Q$ be a regular $2|B|$-gon, and label the sides $\alpha_{b, \text{out}}$, $\alpha_{b, \text{in}}$ in adjacent pairs going anticlockwise, with their common corner labelled $b$, following the cyclic order $\mathit{ord}'$.

Let $T$ be an embedding of $\mathscr{T}$ in $Q$, with the endpoint of the leg labelled $b$ at the corner labelled $b$, so that $T$ intersects the boundary of $Q$ only at the labelled corners. It is straightforward to see that we can pick this embedding so that $T$ has the correct order of half-edges at each vertex -- that is, $\mathit{ord}_T = \mathit{ord}$. We pick our embedding to satisfy this condition. Let $v$ be a vertex of $T$.

Glue the side $\alpha_{b, \text{out}}$ to the side $\alpha_{b, \text{in}}$ by an isometry for each $b \in B$. This yields a sphere $S^2$, and the labelled corners of $Q$ give a subset $B \subset S^2$. The pair $(S^2, B)$ will be the target of the branched cover, and the copy of $T$ inside $(S^2,B)$ will be the embedding.

Take a disjoint union $Q_1 \sqcup \ldots \sqcup Q_d$ of $d$ copies of $Q$, and let $v_i$ be the copy of $v$ inside $Q_i$. We can interpret $\mathit{mon}$ as a set of instructions for gluing the sides of the polygons $Q_i$ to obtain a surface: we glue the side $\alpha_{b, \text{out}}$ of $Q_i$ to the side $\alpha_{b, \text{in}}$ of $Q_j$ by an isometry (taking the corner labelled $b$ to the corner labelled $b$) where $j = \left(\mathit{mon}(b)\right)(i)$. After doing this for all choices of $b$ and $i$, we are left with a closed, oriented surface. If $\mathit{cyc}(a) = (i_1 \ldots i_k) \in \op{Cycles}(\mathit{mon}(b))$, then the corners labelled $b$ of the polygons $Q_{i_1}, \ldots, Q_{i_k}$ all get glued together. We label the resulting point $a$. We will see in a moment that this surface is a connected surface of genus $g$, so we may identify it with $(X,A)$.

\begin{sloppypar}
Let $\overline{f}: Q_1 \sqcup \ldots \sqcup Q_d \rightarrow Q$ be the identity map on each copy of $Q$. This map is compatible with the gluings of source and target, and so descends to a degree-$d$ branched cover ${f: (X, A) \rightarrow (S^2, B)}$. 
\end{sloppypar}

The set of branch points is indeed contained in $B$. (The only other possible branch points are the remaining corners of $Q$, which all get glued to a single point $x \in S^2$, but condition (i.) of Definition \ref{def:decoration} ensures that $f^{-1}(x)$ contains no critical points.) By construction, $f$ realises the portrait $P$: $f$ maps the points of $A$ as prescribed by $\phi$, and if $\mathit{cyc}(a) = (i_1\, \ldots \, i_k) \in \op{Cycles}(\mathit{mon}(b))$ then $a$ maps to $b$ with local degree $\op{rm}_f(a) = k = \op{rm}(a)$. For the same reason, $\op{br}_f(b) = \op{br}(b)$ for each $b \in B$. The Riemann-Hurwitz formula (Definition \ref{def:portrait}, condition (2)) tells us that the genus of the source surface is $g$. Proposition \ref{prop:transitive monodromy group implies connected source} implies that the source surface is connected.

Set $l$ to be the labelling $f^{-1}(v) = \{v_1, \ldots, v_d\}$. Then the branched cover $f$ together with the data $T, v, l$ satisfies $D_f(T, v, l) = (\mathit{ord}, \mathit{mon}, \mathit{cyc})$: it's clear from the construction that the decoration on the legs of $\mathscr{T}$ is correct, and then by Remark \ref{remark:decoration on the legs defines the rest}, (2), the rest of the decoration is also correct.
\end{proof}

Figure \ref{fig:realising a decorated tree example} illustrates the construction for the case of the decorated tree from Figure \ref{fig:decoration first example}.

\begin{figure}[ht!]
\centering
\includegraphics[width=1\textwidth]{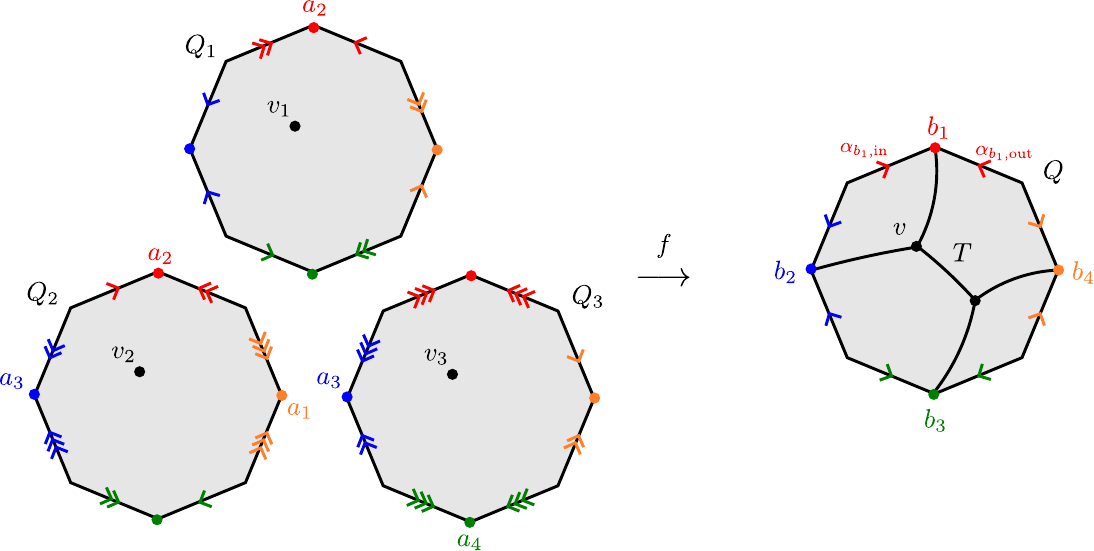}
\caption{Constructing $f, T, v, l$ such that $D_f(T,v,l) = (\mathit{ord}, \mathit{mon}, \mathit{cyc})$ is the decorated tree in Figure \ref{fig:decoration first example}.}
\label{fig:realising a decorated tree example}
\end{figure}

\end{subsection}
\end{section}

\newpage
\begin{section}{Decorated trees from multicurves}\label{sec:decorated trees from multicurves}
In this section, we describe the relationship between multicurves and Hurwitz equivalence classes of decorated trees. We have seen that a stratum of a compactified Hurwitz space can be viewed as an orbit of multicurves (Proposition \ref{prop:HV10 statement}) and as a result, the statements of this section will translate directly to a proof of Theorem \ref{thm:main theorem}.

\begin{subsection}{Multicurves and decorated trees}\label{subsec:multicurves and decorated trees}
Suppose $f: (X,A) \rightarrow (S^2,B)$ is a branched cover and $\Delta$ is a multicurve in $(S^2,B)$. We can assign a $P(f)$-decorated tree to $\Delta$ by picking an embedding of $\mathscr{T}_\Delta$ that is dual to $\Delta$ and feeding it into the decoration function $D_f$, along with a choice of vertex and labelling. Because of the freedom in choosing the embedding, the vertex and the labelling, we can only expect the output decoration to be defined up to Hurwitz equivalence. The next lemma says that this is indeed the case.

\begin{lemma}\label{lemma:different embedded dual trees give equivalent decorations}
\begin{sloppypar}
Let $\Delta$ be a multicurve in $(S^2,B)$. Suppose $T$ and $T'$ are embeddings of $\mathscr{T}_\Delta$ with $\Delta_T$ and $\Delta_{T'}$ isotopic to $\Delta$, that $v$, $v'$ are vertices of $T$, $T'$ and that $l$, $l'$ are labellings of $f^{-1}(v), f^{-1}(v')$, respectively. Then $D_f(T,v,l)$ and $D_f(T', v', l')$ are Hurwitz equivalent. 
\end{sloppypar}
\end{lemma}

\begin{proof}
\begin{sloppypar}
By Proposition \ref{prop:dual trees connected by simple moves}, there exists a sequence of embeddings ${T = T_0}, T_1, \ldots, T_N$ so that $T_i$ differs from $T_{i-1}$ by a braid move and $T_N$ is isotopic to $T'$. By Lemma \ref{lemma:braid move of trees gives braid move of decorations} and Lemma \ref{lemma:changing vertex and labelling of decoration}, the $i$th and $(i-1)$th terms of the sequence ${D_f(T_0, v, l)}, {D_f(T_1, v, l)}, \ldots, {D_f(T_N, v, l)}$ differ by a braid move, and possibly a couple of global conjugations (to change back and forth from the vertex $v$ to the vertex $\op{base}(h)$ when applying Lemma \ref{lemma:changing vertex and labelling of decoration}). But $T_N$ is isotopic to $T'$, so by Lemma \ref{lemma:isotoping T,v,l}, the decoration $D_f(T_N, v, l)$ is equal to a decoration $D_f(T', \widetilde{v}, \widetilde{l})$. One more application of Lemma \ref{lemma:changing vertex and labelling of decoration} tells us that $D_f(T', \widetilde{v}, \widetilde{l})$ is globally conjugate to $D_f(T', v', l')$. We can thus obtain the decoration $D_f(T', v', l')$ from $D_f(T, v, l)$ by a sequence of braid moves and global conjugations, and therefore they are Hurwitz equivalent. 
\end{sloppypar}
\end{proof}

Let $\overline{D}_f$ be the resulting well-defined function from the set of multicurves to the set of Hurwitz equivalence classes of $P(f)$-decorated trees:
\begin{align*}
\overline{D}_f: \op{Mult}_{0,B} &\longrightarrow \op{Stab}_{0,B}(P(f))\\
\Delta &\longmapsto \left[ D_f(T,v,l) \right],
\end{align*}
where $T$ is an embedding of $\mathscr{T}_\Delta$ with $\Delta_T$ isotopic to $\Delta$, $v$ a vertex of $T$ and $l$ a labelling of $f^{-1}(v)$. The next lemma tells us that $\overline{D}_f$ behaves well with respect to containment of multicurves.

\begin{lemma}\label{lemma:Dfbar behaves well with containment}
Suppose $\Delta$ and $\Delta'$ are multicurves with $\Delta' \subset \Delta$. Then $\overline{D}_f(\Delta)$ contracts (as a Hurwitz equivalence class of $P$-decorated trees) to $\overline{D}_f(\Delta')$. 
\end{lemma}

\begin{proof}
Make a choice of $T,v,l$ for $\Delta$. Let $T'$ be a contraction of $T$ along the edges $e$ of $T$ for which $\delta_e \in \Delta \smallsetminus \Delta'$, keeping $v$ fixed throughout. The dual multicurve $\Delta_{T'}$ is isotopic to $\Delta'$. The loops $\overline{\gamma}_h$ used to construct the decorated tree $D_f(T',v,l)$ are homotopic to loops used to construct the decorated tree $D_f(T,v,l)$. It follows that $D_f(T,v,l)$ contracts to $D_f(T',v,l)$. Thus, $\overline{D}_f(\Delta)$ contracts to $\overline{D}_f(\Delta')$. 
\end{proof}

We now come to the main technical proposition, which will allow us to study \textit{orbits} of multicurves. The proof uses the other direction of the `Ikea approach' of \cite{cavalieri2016} to the one in our proof of Proposition \ref{prop:realising decorations}.

\begin{proposition}\label{prop:main prop}
Let $f, f': (X,A) \rightarrow (S^2,B)$ be branched covers and let $\Delta$, $\Delta'$ be multicurves in $(S^2,B)$. Then $\overline{D}_f(\Delta) = \overline{D}_{f'}(\Delta')$ if and only if $f$ and $f'$ are $(A,B)$-Hurwitz equivalent via homeomorphisms $h$, $\widetilde{h}$ with $h(\Delta) = \Delta'$. 
\end{proposition}

\begin{proof}
\begin{sloppypar}
One direction is easy: suppose $f, f'$ are $(A,B)$-Hurwitz equivalent via homeomorphisms $h$, $\widetilde{h}$ with $h(\Delta) = \Delta'$. Make a choice of $T,v,l$ with $\Delta_T$ isotopic to $\Delta$. The embedding $h(T)$ has dual multicurve $\Delta_{h(T)}$ isotopic to $\Delta'$, $h(v)$ is a vertex, and $l \circ \widetilde{h}^{-1}$ a choice of labelling of the preimage of $h(v)$. It follows from the commutative diagram for $(A,B)$-Hurwitz equivalence (\ref{cd:hurwitz equivalence of branched covers}) that $D_f(T,v,l) = D_{f'}(h(T), h(v), l \circ \widetilde{h}^{-1})$. Thus, $\overline{D}_f(\Delta) = \overline{D}_{f'}(h(\Delta)) = \overline{D}_{f'}(\Delta')$.

For the other direction, suppose that $\overline{D}_f(\Delta) = \overline{D}_{f'}(\Delta')$. This means we can make choices ${T, v, l}$ and ${T', v',l'}$ (where $\Delta_T$ is isotopic to $\Delta$ and $\Delta_{T'}$ is isotopic to $\Delta'$) so that ${D_{f}(T,v,l)}$ and ${D_{f'}(T',v',l')}$ are Hurwitz equivalent decorations. Changing $T'$ by braid moves and postcomposing $l'$ with a permutation, if necessary, we can assume without loss of generality that the two decorations are equal, say ${D_{f}(T, v, l)} = {D_{f'}(T', v', l')} = {(\mathit{ord}, \mathit{mon}, \mathit{cyc})}.$

Pick a point $x$ in $S^2 \smallsetminus T$ and draw an arc $\alpha_b$ from $x$ to $b$ for each point $b \in B$. Choose these arcs so that their interiors are pairwise disjoint and do not intersect $T$. Cut along the arcs $\alpha_b$ to obtain a $2|B|$-gon $Q$ (that is, a topological disk whose boundary is subdivided into $2|B|$ `sides' that meet at `corners') containing a copy of $T$, with a regluing map $g: Q \rightarrow S^2$. For each $b\in B$, there are two adjacent sides of $Q$ that reglue to $\alpha_b$ -- label these $\alpha_{b, \text{out}}$ and $\alpha_{b, \text{in}}$ in that order going anticlockwise. In the same way, pick a point $x'$ in $S^2 \smallsetminus T'$ and cut along arcs $\alpha'_b$ to obtain a $2|B|$-gon $Q'$ that contains a copy of $T'$, with a regluing map $g':Q' \rightarrow S^2$. Label the sides $\alpha'_{b, \text{out}}$ and $\alpha'_{b, \text{in}}$.

The cyclic order of the arcs $\alpha_b$ leaving $x$ -- and therefore the cyclic order of the side labels of $Q$ -- is determined by $\mathit{ord}$. The same is true for the side labels of $Q'$. Thus, the side labels of $Q$ and $Q'$ have the same cyclic order, and there is a homeomorphism $\overline{h}: Q \rightarrow Q'$ that takes $\alpha_{b, \text{out}}$ to $\alpha_{b, \text{out}}'$ and takes $\alpha_{b, \text{in}}$ to $\alpha_{b, \text{in}}'$.

We also wish this homeomorphism to (1) be compatible with the regluing maps ${g:Q \rightarrow S^2}$ and ${g':Q' \rightarrow S^2}$, and (2) take $T$ to $T'$. It is straightforward to see that we can choose $\overline{h}$ to satisfy these properties. Gluing back up the source and target, $\overline{h}$ gives a homeomorphism $h: (S^2,B) \rightarrow (S^2,B)$ such that $h(T) = T'$. We now show that this homeomorphism lifts.

Take $d$ copies of the polygon $Q$, and label them $Q_1, \ldots, Q_d$. Because the interior of $Q$ is simply connected and $f$ is a degree-$d$ cover away from $B$, the preimage $f^{-1}(g(\op{int}(Q)))$ is the disjoint union of $d$ copies of $g(\op{int}(Q))$. Let $g_i = f^{-1} \circ g: \op{int}(Q_i) \rightarrow X$ be the branch that glues $\op{int}(Q_i)$ to the subset of $X$ containing $v_i$. The gluing map can be extended to the entire polygon, $g_i: Q_i \rightarrow X$. Similarly, define gluing maps $g_i' = f^{-1} \circ g': Q_i' \rightarrow X$.

This lets us view the surface $X$ as a union of the polygons $Q_i$, glued along their sides. The gluing instructions are encoded in the decoration $(\mathit{ord}, \mathit{mon}, \mathit{cyc})$: the side $\alpha_{b, \text{out}}$ of $Q_i$ is glued to the side $\alpha_{b, \text{in}}$ of $Q_j$, where $j=(\mathit{mon}(b))(i)$. We can similarly view $X$ as a union of the polygons $Q_i'$, glued along their sides with the same gluing instructions.

Because the gluing instructions are the same, the maps $\overline{h}: Q_i \rightarrow Q_i'$ descend to a homeomorphism $\widetilde{h}: X \rightarrow X$. By definition, $\widetilde{h}$ is a lift of $h$.

We claim that $\widetilde{h}$ fixes the points of $A \subset X$. Take $a \in A$. If $\mathit{cyc}(a) = (i_1\,\ldots\,i_k) \in \op{Cycles}(\mathit{mon}(b))$, then the corner labelled $b$ in each of the polygons $Q_{i_1}, \ldots, Q_{i_k}$ glues to the point $a \in X$. Similarly, the corner labelled $b$ in each of the polygons $Q_{i_1}', \ldots, Q_{i_k}'$ glues to $a$. $\widetilde{h}$ descends from the maps $\overline{h}: Q_i \rightarrow Q_i'$, and so $\widetilde{h}(a) = a$. 
\end{sloppypar}
\end{proof}

This proposition immediately implies that the function $\overline{D}_f$ outputs the same Hurwitz equivalence class of decorated trees on orbits of multicurves under the action of $\op{LMod}_f$, and that the resulting map on orbits is injective.

\begin{corollary}\label{cor:corollary of main proposition}
Let $f:(X,A) \rightarrow (S^2,B)$ be a branched cover and let $\Delta$, $\Delta'$ be multicurves in $(S^2,B)$. Then $\op{LMod}_f(\Delta) = \op{LMod}_f(\Delta')$ if and only if $\overline{D}_f(\Delta) = \overline{D}_f(\Delta')$.
\end{corollary}
\end{subsection}

\begin{subsection}{Orbits of multicurves and decorations} 
\begin{sloppypar}
Let $X$ be a compact surface of genus $g$, and let $A \subset X$ and $B \subset S^2$ be finite subsets. Fix a portrait ${P = (g,d,B,A, \phi:A \rightarrow B,\op{br},\op{rm})}$, and pick representatives $f_1, \ldots, f_N: (X,A) \rightarrow (S^2,B)$ of the $(A,B)$-Hurwitz equivalence classes of branched covers realising $P$. Corollary \ref{cor:corollary of main proposition} tells us that the following function on \textit{orbits} of multicurves is well-defined.
\begin{align*}
\widetilde{D}_{f_1, \ldots, f_N}: \bigsqcup_{i=1}^N \op{Mult}_{0,B} \big/ \op{LMod}_{f_i} &\longrightarrow \op{Stab}_{0,B}(P)\\
\op{LMod}_{f_i}(\Delta) &\longmapsto \overline{D}_{f_i}(\Delta).
\end{align*} 
\end{sloppypar}

\begin{proposition}\label{prop:Dbar is a bijection}
The function $\widetilde{D}_{f_1, \ldots, f_N}$ is a bijection. 
\end{proposition}

\begin{proof}
We abbreviate $\widetilde{D}_{f_1, \ldots, f_N}$ to $\widetilde{D}$ throughout the proof.

\textit{Injectivity.} Suppose $\widetilde{D}(\op{LMod}_{f_i}(\Delta)) =  \widetilde{D}(\op{LMod}_{f_j}(\Delta'))$. Then $\overline{D}_{f_i}(\Delta) = \overline{D}_{f_j}(\Delta')$. By Proposition \ref{prop:main prop}, $f_i$ and $f_j$ are $(A,B)$-Hurwitz equivalent by homeomorphisms $h, \widetilde{h}$ with $h(\Delta) = \Delta'$. Because we chose one representative for each $(A,B)$-Hurwitz equivalence class, $f_i = f_j$ and $h$ is a liftable homeomorphism. Thus, $\Delta' \in \op{LMod}_{f_i}(\Delta)$.

\textit{Surjectivity.} Let $(\mathit{ord}, \mathit{mon}, \mathit{cyc})$ be a $P$-decorated tree. By Proposition \ref{prop:realising decorations}, there exists a tree $T$, a vertex $v$ and a labelling $l$ such that $D_f(T,v,l) = (\mathit{ord}, \mathit{mon}, \mathit{cyc})$. The branched cover $f$ is $(A,B)$-Hurwitz equivalent to one of the representatives $f_i$, via some homeomorphisms $h,\widetilde{h}$. By Proposition \ref{prop:main prop}, $\overline{D}_{f_i}(h(\Delta_T)) = \overline{D}_f(\Delta_T) = [(\mathit{ord}, \mathit{mon}, \mathit{cyc})]$, and so $\widetilde{D}(\op{LMod}_{f_i}(h(\Delta)) = [(\mathit{ord}, \mathit{mon}, \mathit{cyc})]$. 
\end{proof}
\end{subsection}

A proof of the main theorem follows immediately.

\begin{repeatmaintheorem}
Let $\mathcal{H}$ be a Hurwitz space that parametrises maps to $\mathbb{P}^1$, with $P$ its defining portrait. The irreducible strata of the compactification $\overline{\mathcal{H}}$ are in bijection with the set $\op{Stab}_{0,B}(P)$ of Hurwitz equivalence classes of $P$-decorated trees. Furthermore, containment of closures of strata in $\overline{\mathcal{H}}$ corresponds to edge contraction of $P$-decorated trees.
\end{repeatmaintheorem}

\begin{proof}
Because we have identified the strata of $\overline{\mathcal{H}}$ as orbits of multicurves (Proposition \ref{prop:HV10 statement}), with containment of closures of strata descending from containment of multicurves, the main theorem now follows as an exact translation of the previous proposition.
\end{proof}
\end{section}

\begin{remark}
\begin{sloppypar}
The bijection of Theorem \ref{thm:main theorem} is canonical. Although we choose representative branched covers $f_1, \ldots, f_N$ realising $P$ to build $\widetilde{D}_{f_1, \ldots, f_N}$, the bijection from the set of strata of $\overline{\mathcal{H}}$ to $\op{Stab}_{0,B}(P)$ is independent of these choices: suppose the stratum $\mathcal{S}$ of $\overline{\mathcal{H}}$ is identified with $\op{LMod}_{f_i}(\Delta)$ under Proposition \ref{prop:HV10 statement}. Then $\mathcal{S}$ is sent to $\overline{D}_{f_i}(\Delta) \in \op{Stab}_{0,B}(P)$. Now, suppose $f_i$ and $f_i'$ are $(A,B)$-Hurwitz equivalent via homeomorphisms $h, \widetilde{h}$. Choosing $f_i'$ instead of $f_i$, it follows from Lemma \ref{lemma:pif and pifprime} that the stratum $\mathcal{S}$ is identified with $\op{LMod}_{f_i'}(h(\Delta))$ under Proposition \ref{prop:HV10 statement} -- so $\mathcal{S}$ is sent to $\overline{D}_{f_i'}(h(\Delta)) \in \op{Stab}_{0,B}(P)$. Indeed, $\overline{D}_{f_i}(\Delta) = \overline{D}_{f_i'}(h(\Delta))$ by Proposition \ref{prop:main prop}.
\end{sloppypar}
\end{remark}

\newpage
\begin{section}{The strata of \texorpdfstring{$\overline{\mathcal{H}}$}{H} under the target and source maps}\label{sec:strata of H under target and source}

Let $X$ be a compact surface of genus $g$, and let $A \subset X$ and $B \subset S^2$ be finite subsets with $2 - 2g - |A| < 0$ and $|B| \geq 3$. Fix a portrait ${P = (g,d,B,A, \phi:A \rightarrow B,\op{br},\op{rm})}$, with $\mathcal{H}$ the associated Hurwitz space. Recall that the natural target and source maps on $\mathcal{H}$ extend to
$$\pi_B: \overline{\mathcal{H}} \rightarrow \overline{\mathcal{M}}_{0,B}, \quad\quad \pi_A: \overline{\mathcal{H}} \rightarrow \overline{\mathcal{M}}_{g,A},$$
respectively (Section \ref{subsec:HV10}). In this section, we explain the connection to combinatorial admissible covers and we explicitly describe how the strata of $\overline{\mathcal{H}}$ map under $\pi_B$ and $\pi_A$ in terms of $P$-decorated trees.

\begin{subsection}{Combinatorial admissible covers}\label{subsec:relation to combinatorial admissible covers}
\begin{sloppypar}
A \textit{combinatorial admissible cover} is a morphism of graphs ${\Theta: \mathscr{G} \rightarrow \mathscr{T}}$ together with a labelling of the edges and legs of $\mathscr{G}$ by positive-integer \textit{expansion factors}\footnote{An \textit{admissible cover} \cite{harris1982kodaira} of nodal Riemann surfaces induces a combinatorial admissible cover: the graph morphism is the morphism of dual graphs, and the expansion factors on legs and edges are the local degrees at marked points and nodes, respectively. A combinatorial admissible cover that arises in this way is called \textit{realisable}. A combinatorial admissible cover is realisable if and only if it is induced by a $P$-decorated tree, in the sense of Section \ref{subsubsec:combinatorial admissible covers from decorated trees}.}. The data must also satisfy a \textit{harmonicity} condition \cite[Section 2.2.1]{cavalieri2016tropicalizing}, but this is automatic for the combinatorial admissible covers that we consider.
\end{sloppypar}

\begin{subsubsection}{Combinatorial admissible covers from $P$-decorated trees}\label{subsubsec:combinatorial admissible covers from decorated trees}
A Hurwitz equivalence class of $P$-decorated trees $\theta \in \op{Stab}_{0,B}(P)$ induces a combinatorial admissible cover $\Theta: \mathscr{G} \rightarrow \mathscr{T}$ as follows. Suppose $\theta = \overline{D}_f(\Delta) \in \op{Stab}_{0,B}(P)$, where $\Delta$ is a multicurve on $(S^2,B)$ and ${f:(X,A) \rightarrow (S^2,B)}$ is a branched cover that realises $P$ (by Proposition \ref{prop:realising decorations}, such a $\Delta$ and $f$ exist). Then
\begin{itemize}
\item $\mathscr{T}$ is the $B$-marked dual tree of $\Delta$. 
\item $\mathscr{G}$ is the $A$-marked dual graph of $f^{-1}(\Delta)$ (a union of simple closed curves in $X$). 
\item $\Theta: \mathscr{G} \rightarrow \mathscr{T}$ is the natural morphism of dual graphs induced by $f$. 
\item The expansion factor on an edge of $\mathscr{G}$ is the degree with which the corresponding curve of $f^{-1}(\Delta)$ maps under $f$, and the expansion factor on the leg $a \in A$ is the local degree $\op{ram}_f(a)$.
\end{itemize}
\end{subsubsection}
\end{subsection}

\begin{subsection}{The target and source maps on the strata of \texorpdfstring{\boldmath $\overbar{\mathcal{H}}$}{H}}\label{subsec:action of piA and piB on strata} 
Suppose $\theta \in \op{Stab}_{0,B}(P)$ induces the combinatorial admissible cover $\Theta: \mathscr{G} \rightarrow \mathscr{T}$.

\begin{subsubsection}{The target map} The target map $\pi_B$ takes the stratum of $\overline{\mathcal{H}}$ indexed by $\theta$ surjectively to the stratum of $\overline{\mathcal{M}}_{0,B}$ indexed by $\mathscr{T}$. 
\end{subsubsection}

\begin{subsubsection}{The source map}
The source map $\pi_A$ takes the stratum of $\overline{\mathcal{H}}$ indexed by $\theta$ to the stratum of $\overline{\mathcal{M}}_{g,A}$ indexed by $\mathscr{G}'$, where $\mathscr{G}'$ is the stabilisation of $\mathscr{G}$. The restriction of $\pi_A$ to this stratum is not necessarily surjective. We now explain how to construct $\mathscr{G}$ and $\mathscr{G}'$ directly, given a $P$-decorated tree $\theta$. Suppose $\theta = [(\mathit{ord}, \mathit{mon}, \mathit{cyc})]$.
\begin{description}[leftmargin=3em, style=nextline]
\item[$V(\mathscr{G})$] is the set of pairs $(v, O)$, where $v$ is a vertex of $\mathscr{T}$ and $O$ is an orbit of $\{1, \ldots, d\}$ under the action of the subgroup $\langle \mathit{mon}(h_1), \ldots, \mathit{mon}(h_k) \rangle \leq S_d$, with $h_1, \ldots, h_k$ the half-edges of $\mathscr{T}$ incident with $v$. The vertex $(v,O)$ receives a \textit{genus weight} $g_{(v,O)}$: it is the unique nonnegative integer that satisfies the \textit{local} Riemann-Hurwitz equation
$$\sum_{c} (\op{len}(c) - 1) - 2g_{(v,O)} = 2|O| - 2,$$ 
where the sum is over cycles $c$ of $\mathit{mon}(h_1), \ldots, \mathit{mon}(h_k)$ with $c \subset O$.
\item[$H(\mathscr{G})$] consists of cycles of permutations (excluding the cycles of leg permutations that are not labelled by $A$):
$$H(\mathscr{G}) \subset \bigsqcup_{h \in H(\mathscr{T})} \op{Cycles}(\mathit{mon}(h)).$$
If $c$ is a cycle of $\mathit{mon}(h)$ with $\op{base}(h)=v$ and $c \subset O$, then $c$ is a half-edge incident with $(v, O)$. If $c$ and $c'=c^{-1}$ are cycles on two half-edges of $\mathscr{T}$ that form an edge in $\mathscr{T}$, then the half-edges $c$ and $c'$ form an edge in $\mathscr{G}$. If $c$ is labelled $\mathit{cyc}(c) = a$, then the half-edge $c$ is a leg labelled $a$. 
\end{description}
The graph $\mathscr{G}$ is not necessarily stable -- we set $\mathscr{G}'$ to be its stabilisation (see Section \ref{subsubsec:stable graphs}). The graph $\mathscr{G}' \in \op{Stab}_{g,A}$ is independent of the representative $(\mathit{ord}, \mathit{mon}, \mathit{cyc})$. Figure \ref{fig:source graph construction} illustrates the construction of $\mathscr{G}$ and $\mathscr{G}'$ from a $P$-decorated tree.
\end{subsubsection}

\begin{figure}[ht!]
\centering
\includegraphics[scale=1]{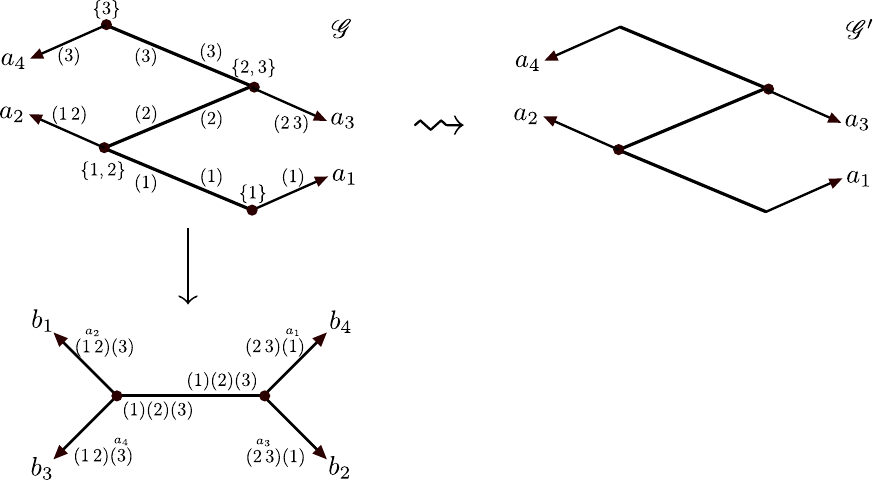}
\caption{Suppose $\theta$ is represented by the $P$-decorated tree at bottom-left. Half-edges of $\mathscr{G}$ are cycles of permutations decorating $\mathscr{T}$. The stabilisation $\mathscr{G}'$ has two vertices, one edge and four legs. All genus weights of vertices are zero.}
\label{fig:source graph construction}
\end{figure}

\end{subsection}
\end{section}

\newpage
\begin{section}{Tropical Hurwitz space}\label{sec:Htrop}
\begin{sloppypar}
In this section, given a Hurwitz space $\mathcal{H}$, we construct an extended cone complex $\overline{\mathcal{H}}^\text{trop}$ that has one $k$-dimensional cone for each codimension-$k$ stratum of $\overline{\mathcal{H}}$, with containment of cones given by containment of closures of strata. We identify this extended cone complex with $\overline{\Sigma}(\overline{\mathcal{H}})$, the skeleton of the Berkovich analytification of $\overline{\mathcal{H}}$. We also define tropicalisations
$${\pi_B^\text{trop}: \overline{\mathcal{H}} \rightarrow \overline{\mathcal{M}}_{0,B}^\text{trop}},\quad\quad {\pi_A^\text{trop}: \overline{\mathcal{H}} \rightarrow \overline{\mathcal{M}}_{g,A}^\text{trop}}$$
of the target and source maps.
\end{sloppypar}

Throughout, let $X$ be a compact surface of genus $g$, and let $A \subset X$ and $B \subset S^2$ be finite subsets with $2 - 2g - |A| < 0$ and $|B| \geq 3$. Fix a portrait ${P = (g,d,B,A, \phi:A \rightarrow B,\op{br},\op{rm})}$ and let $\mathcal{H}$ be the associated Hurwitz space.

\begin{subsection}{Tropical Hurwitz space \texorpdfstring{\boldmath $\overline{\mathcal{H}}^\text{trop}$}{Htrop}}\label{subsec:Htrop definition}
We start by defining the extended cones, one for each stratum of $\overline{\mathcal{H}}$. Suppose $\theta \in \op{Stab}_{0,B}(P)$ is a Hurwitz equivalence class of $P$-decorations of $\mathscr{T} \in \op{Stab}_{0,B}$. The extended cone associated to $\theta$ is ${\overline{\sigma}_\theta = (\mathbb{R}_{\geq 0} \cup \{\infty\})^{E(\mathscr{T})}}$, together with its standard integer lattice.

Next, we describe the gluings of faces. By Lemma \ref{lemma:contraction of Heq classes is well defined}, contraction of $\mathscr{T}$ yields a well-defined Hurwitz equivalence class of $P$-decorations $\theta'$ of the contracted tree $\mathscr{T}'$. The inclusion $E(\mathscr{T}') \hookrightarrow E(\mathscr{T})$ induces an isomorphism from $\overline{\sigma}_{\theta'}$ to a face of $\overline{\sigma}_\theta$. We set 
$$\overline{\mathcal{H}}^\text{trop} \,\,= \bigsqcup_{\theta \in \op{Stab}_{0,B}(P)} \overline{\sigma}_\theta \,\,\,\big/ \sim,$$
where faces are glued as prescribed by edge contraction\footnote{Replacing extended cones with standard cones $\sigma_\theta = (\mathbb{R}_{\geq 0})^{E(\mathscr{T})}$ in the construction yields a cone complex $\mathcal{H}^\text{trop}$ which is exactly the cone over the boundary complex of $\overline{\mathcal{H}}$.}.

The inclusion $\mathcal{H} \subset \overline{\mathcal{H}}$ is a toroidal embedding of Deligne-Mumford stacks\footnote{We work with $\overline{\mathcal{H}}$ as the complex orbifold associated to the usual Deligne-Mumford stack, but the stratifications of the orbifold and the stack are the same and we may apply the results of \cite{abramovich2015tropicalization}.}. In this setting, Abramovich, Caporaso and Payne interpret the \textit{skeleton} $\overline{\Sigma}(\overline{\mathcal{H}})$ as the tropicalisation of $\overline{\mathcal{H}}$. The skeleton lives inside the \textit{Berkovich analytification} $\overline{\mathcal{H}}^\text{an}$ as the image of a canonically defined retraction map $p_{\overline{\mathcal{H}}}: \overline{\mathcal{H}}^\text{an} \rightarrow \overline{\mathcal{H}}^\text{an}$. Combining Theorem \ref{thm:main theorem} with Proposition 6.2.6 of \cite{abramovich2015tropicalization} and analysis carried out by Cavalieri, Markwig and Ranganathan for the case of $\overline{\mathcal{H}}$ \cite[Section 4]{cavalieri2016tropicalizing}, we obtain the following corollary.

\begin{repeatskeletoncorollary}
Let $\mathcal{H}$ be a Hurwitz space parametrising maps to $\mathbb{P}^1$. The extended cone complex $\overline{\mathcal{H}}^\text{trop}$ and the skeleton of the Berkovich analytification $\overline{\Sigma}(\overline{\mathcal{H}})$ are isomorphic as extended cone complexes with integral structure.
\end{repeatskeletoncorollary}

\begin{proof}
By Proposition 6.2.6 of \cite{abramovich2015tropicalization}, the skeleton decomposes as a union of generalised extended cones -- one for each irreducible stratum: 
$$\overline{\Sigma}(\overline{\mathcal{H}}) \,\, \cong \bigsqcup_{\theta \in \op{Stab}_{0,B}(P)} (\overline{\sigma}_\theta \,/\, H_\theta) \,\,\, \big/ \sim,$$
with faces glued as prescribed by edge contraction. The group $H_\theta$ is the \textit{monodromy group} of a point in the stratum $\theta$ \cite[Definition 6.2.2]{abramovich2015tropicalization}. By Section 4 of \cite{cavalieri2016tropicalizing} $H_\theta$ is the group $\op{Aut}(\Theta)$ of automorphisms of the combinatorial admissible cover $\Theta$ induced by $\theta$ (see Section \ref{subsec:relation to CMR}), but because the $B$-marked stable tree $\mathscr{T}$ has no non-trivial automorphisms, $\op{Aut}(\Theta)$ acts trivially on $\overline{\sigma}_\theta$.
\end{proof}
\end{subsection}

\begin{subsection}{Tropical moduli space \texorpdfstring{\boldmath $\overline{\mathcal{M}}_{g,A}^\text{trop}$}{MgAtrop}}\label{subsec:MgAtrop}
By a \textit{metrisation} of an $A$-marked graph $\mathscr{G}$, we mean an assignment $m_{\mathscr{G}}$ of a length in $\mathbb{R}_{\geq 0} \cup \{\infty\}$ to each edge, and a length of $\infty$ to each leg, up to automorphisms of the graph. A metrisation can be identified with a point of the generalised extended cone ${\overline{\sigma}_{\mathscr{G}} = (\mathbb{R}_{\geq 0} \cup \{\infty\})^{E(\mathscr{G})} / \op{Aut}(\mathscr{G})}$. A \textit{tropical curve} $(\mathscr{G}, m_{\mathscr{G}})$ is a graph $\mathscr{G}$ together with a metrisation $m_{\mathscr{G}}$.

\begin{sloppypar}
The \textit{tropical moduli space} $\overline{\mathcal{M}}_{g,A}^\text{trop}$ is a generalised extended cone complex that parametrises tropical curves $(\mathscr{G}, m_{\mathscr{G}})$, where $\mathscr{G}$ is an $A$-marked genus-$g$ stable graph. It has one generalised extended cone ${\overline{\sigma}_{\mathscr{G}} = (\mathbb{R}_{\geq 0} \cup \{\infty\})^{E(\mathscr{G})} / \op{Aut}(\mathscr{G})}$ for each isomorphism class of graph ${\mathscr{G} \in \op{Stab}_{g,A}}$, and cones are glued as prescribed by edge contraction of graphs. The tropicalisation $\overline{\mathcal{M}}_{g,A}^\text{trop}$ is isomorphic to the skeleton $\overline{\Sigma}(\overline{\mathcal{M}}_{g,A})$ of the Berkovich analytification, as generalised extended cone complexes with integral structure \cite[Theorem 1.2.1]{abramovich2015tropicalization}.
\end{sloppypar}
\end{subsection}

\begin{subsection}{Tropical admissible covers}
A \textit{tropical admissible cover} is a combinatorial admissible cover $\Theta: \mathscr{G} \rightarrow \mathscr{T}$ together with metrisations $m_{\mathscr{G}}$ and $m_{\mathscr{T}}$ of the source and target graphs, respectively, so that if the edge $e$ maps to $\Theta(e)$ with expansion factor $c_e$, then ${c_e \cdot m_{\mathscr{G}}(e) = m_{\mathscr{T}}(\Theta(e))}$. We write $\Theta: (\mathscr{G}, m_{\mathscr{G}}) \rightarrow (\mathscr{T}, m_{\mathscr{T}})$ for this tropical admissible cover.

Given a combinatorial admissible cover $\Theta: \mathscr{G} \rightarrow \mathscr{T}$ and a metrisation $m_{\mathscr{T}}$ of $\mathscr{T}$, there is a unique metrisation $m_{\mathscr{G}}$ of $\mathscr{G}$ so that $\Theta: (\mathscr{G}, m_{\mathscr{G}}) \rightarrow (\mathscr{T}, m_{\mathscr{T}})$ is a tropical admissible cover. The metrisation $m_{\mathscr{G}}$ is defined by the equations $c_e \cdot m_{\mathscr{G}}(e) = m_{\mathscr{T}}(\Theta(e))$, $e \in E(\mathscr{G})$.

\begin{subsubsection}{Tropical admissible covers from $P$-decorated trees}\label{subsubsection:tropical admissible covers from decorated trees}
\begin{sloppypar}
Suppose $\theta \in \op{Stab}_{0,B}(P)$ is an equivalence class of $P$-decorations of $\mathscr{T}$, and that ${(x_e)_{e \in E(\mathscr{T})}}$ is a point of the cone ${\overline{\sigma}_\theta} \subset \overline{\mathcal{H}}^\text{trop}$. The point induces a tropical admissible cover ${\Theta: (\mathscr{G}, m_{\mathscr{G}}) \rightarrow (\mathscr{T}, m_{\mathscr{T}})}$ as follows.
\begin{itemize}
\item $\Theta: \mathscr{G} \rightarrow \mathscr{T}$ is the combinatorial admissible cover induced by $\theta$ (Section \ref{subsubsec:combinatorial admissible covers from decorated trees}).
\item $m_{\mathscr{T}}$ assigns a length $L(e)\,x_e$ to the edge $e$, where $L(e)$ is the least common multiple of the expansion factors on the edges of $\mathscr{G}$ that map to $\mathscr{T}$. (The factor $L(e)$ comes from examining the branch map in local coordinates \cite[Section 5.4]{cavalieri2016tropicalizing}.)
\item $m_{\mathscr{G}}$ is the unique metrisation of $\mathscr{G}$ so that ${\Theta: (\mathscr{G}, m_{\mathscr{G}}) \rightarrow (\mathscr{T}, m_{\mathscr{T}})}$ is a tropical admissible cover.
\end{itemize}
\end{sloppypar}
\end{subsubsection}
\end{subsection}

\begin{subsection}{The tropical target and source maps}\label{subsec:tropical src and tgt maps} 
Suppose $\theta \in \op{Stab}_{0,B}(P)$, and consider the point $(x_e)_{e \in E(\mathscr{T})}$ of the cone $\overline{\sigma}_\theta$. Let ${\Theta: (\mathscr{G}, m_{\mathscr{G}}) \rightarrow (\mathscr{T}, m_{\mathscr{T}})}$ be the tropical admissible cover that is induced by the point.

\begin{subsubsection}{The tropical target map}
The tropical target map $\pi_B^\text{trop}: \overline{\mathcal{H}}^\text{trop} \rightarrow \overline{\mathcal{M}}_{0,B}^\text{trop}$ takes the cone $\overline{\sigma}_\theta$ to the cone $\overline{\sigma}_{\mathscr{T}}$: 
\begin{align*}
\pi_B^\text{trop}: \overline{\sigma}_\theta &\longrightarrow \overline{\sigma}_{\mathscr{T}}\\
(x_e)_{e \in E(\mathscr{T})} & \longmapsto m_\mathscr{T} = (L(e)\, x_e)_{e \in E(\mathscr{T})}.
\end{align*}

\end{subsubsection}

\begin{subsubsection}{The tropical source map}
The tropical source map $\pi_A^\text{trop}: \overline{\mathcal{H}}^\text{trop} \rightarrow \overline{\mathcal{M}}_{g,A}^\text{trop}$ takes the cone $\overline{\sigma}_\theta$ to the cone $\overline{\sigma}_{\mathscr{G}'}$, where $\mathscr{G}'$ is the stabilisation of $\mathscr{G}$, 
\begin{align*}
\pi_A^\text{trop}: \overline{\sigma}_\theta &\longrightarrow \overline{\sigma}_{\mathscr{G}'}\\
(x_e)_{e \in E(\mathscr{T})} &\longmapsto m_{\mathscr{G}'}.
\end{align*}
The metrisation $m_{\mathscr{G'}}$ of $\mathscr{G}'$ is induced by $m_{\mathscr{G}}$: recall that an edge $e$ of $\mathscr{G}'$ can be identified with a path of edges of $\mathscr{G}$; the length $m_{\mathscr{G}'}(e)$ of $e$ is simply the length of the path under $m_{\mathscr{G}}$.
\end{subsubsection}
\end{subsection}

\begin{subsection}{Relation to \texorpdfstring{\boldmath $\overline{\mathcal{H}}^\text{trop}_\text{CMR}$}{HtropCMR}}\label{subsec:relation to CMR}
\begin{sloppypar}
The extended cone complex $\overline{\mathcal{H}}^\text{trop}_{\text{CMR}}$ parametrises tropical admissible covers: a combinatorial admissible cover\footnote{Strictly speaking, $\Theta$ gives a cone $\overline{\sigma}_\Theta$ of $\overline{\mathcal{H}}_\text{CMR}^\text{trop}$ only if $\Theta$ is realisable -- that is, if $\Theta$ is induced by some $\theta \in \op{Stab}_{0,B}(P)$, in the sense of Section \ref{subsubsec:combinatorial admissible covers from decorated trees}.} $\Theta: \mathscr{G} \rightarrow \mathscr{T}$ gives a cone $\overline{\sigma}_\Theta = (\mathbb{R}_{\geq 0} \cup \{\infty\})^{E(\mathscr{T})}$ of $\overline{\mathcal{H}}^\text{trop}_\text{CMR}$. The construction of Section \ref{subsubsection:tropical admissible covers from decorated trees} that assigns a tropical admissible cover to a point of $\overline{\sigma}_\theta \subset \overline{\mathcal{H}}^\text{trop}$ induces a forgetful map of cone complexes 
\begin{align*}
\mathit{forget}: \overline{\mathcal{H}}^\text{trop} &\longrightarrow \overline{\mathcal{H}}^\text{trop}_\text{CMR}.
\end{align*}
Our tropical source and target maps $\pi_{B}^\text{trop}$, $\pi_{A}^\text{trop}$ factor as ${\pi_{B,\text{CMR}}^\text{trop} \circ \mathit{forget}}$, ${\pi_{A,\text{CMR}}^\text{trop} \circ \mathit{forget}}$, respectively, where $\pi_{B,\text{CMR}}^\text{trop}$, $\pi_{A,\text{CMR}}^\text{trop}$ are the tropical target and source maps defined in \cite{cavalieri2016tropicalizing}. 
\end{sloppypar}
\end{subsection}

\begin{subsection}{Tropical Thurston theory}\label{subsec:tropical thurston theory}
To conclude, we make a connection with tropical Thurston theory \cite{ramadas2024thurston}. Suppose that $f: S^2 \rightarrow S^2$ is a branched cover from the sphere to itself whose \textit{postcritical set} -- the set $P_f = \bigcup_{n \geq 1} f^{\circ n}(C_f)$ of iterates of the set $C_f$ of critical points of $f$ -- is finite. Such a map is called a \textit{Thurston map}. Set $A=B=P_f$.

Recall the map $\pi_f: \mathcal{T}_{0,B} \rightarrow \mathcal{H}_f$ from Section \ref{subsubsec:marked holomorphic maps from branched covers}. Ramadas defines a tropicalisation of this map \cite[Definition 7.4]{ramadas2024thurston},
$$\pi_{f,\text{CMR}}^\text{trop}: \mathcal{T}_{0,B}^\text{trop} \longrightarrow \mathcal{H}_{f,\text{CMR}}^\text{trop}.$$
Here, $\mathcal{T}_{0,B}^\text{trop}$ is the cone over the curve complex: a multicurve $\Delta$ gives a cone $\sigma_\Delta = (\mathbb{R}_{\geq  0})^\Delta$ of $\mathcal{T}_{0,B}^\text{trop}$ and cones are glued along faces as prescribed by containment of multicurves.

We reformulate this definition using our tropical Hurwitz space. Define
\begin{align*}
\pi_f^\text{trop}: \mathcal{T}_{0,B}^\text{trop} &\longrightarrow \mathcal{H}_{f}^\text{trop}\\
\sigma_\Delta &\longrightarrow \sigma_\theta\\
(x_\delta)_{\delta \in \Delta} &\longmapsto \left(\frac{1}{L(e)} x_{\delta_e}\right)_{e \in E(\mathscr{T})} 
\end{align*}
where $\theta = \overline{D}_f(\Delta)$, and $\delta_e$ is the simple closed curve of $\Delta$ corresponding to $e$. The next proposition is immediate.

\begin{proposition}
The map $\pi_{f,\text{CMR}}^\text{trop}$ factors as 
$$\pi_{f,\text{CMR}}^\text{trop} = \mathit{forget} \circ \pi_f^\text{trop}.$$
\end{proposition}

Finally, we reformulate the definition of the tropical moduli space correspondence to obtain the `true' tropical moduli space correspondence \cite[Remark 7.1]{ramadas2024thurston}.

\begin{repeatdeftropicalcorrespondence}
The \textit{tropical moduli space correspondence} is
$$(\pi_A^\text{trop}|_{\mathcal{H}_f^\text{trop}}) \circ (\pi_B^\text{trop}|_{\mathcal{H}_f^\text{trop}})^{-1}: \mathcal{M}_{0,P_f}^\text{trop} \rightrightarrows \mathcal{M}_{0,P_f}^\text{trop},$$
where $\pi_B^\text{trop}$ and $\pi_A^\text{trop}$ are the tropical target and source maps on $\mathcal{H}^\text{trop}$, respectively. 
\end{repeatdeftropicalcorrespondence}
\end{subsection}
\end{section}

\newpage
\appendix
\begin{section}{Proof of Proposition \ref{prop:dual trees connected by simple moves}}\label{sec:appendix}

We now provide a proof of Proposition \ref{prop:dual trees connected by simple moves}. As a reminder:

\begin{repeatpropdualtrees}
Suppose $T$ and $T'$ are embeddings of a $B$-marked stable tree $\mathscr{T}$ in $(S^2,B)$ such that $\Delta_T$ and $\Delta_{T'}$ are isotopic. Then there exists a sequence of embeddings of $\mathscr{T}$
$$T=T_0,\,\, T_1,\,\, \ldots,\,\, T_N$$
such that $T_i$ differs from $T_{i-1}$ by a braid move for $i=1, \ldots, N$, and $T_N$ is isotopic to $T'$.
\end{repeatpropdualtrees}

We start with two technical lemmas. Roughly, the first lemma states that if two embeddings $T$ and $T'$ become isotopic after pinching a common dual multicurve $\Delta$, then $T$ and $T'$ only differ by Dehn twists about $\Delta$.

\begin{lemma}\label{lemma:lifting a pinched tree}
Let $T$ and $T'$ be embeddings of a $B$-marked stable tree $\mathscr{T}$ in $(S^2,B)$, and suppose there exists a common dual multicurve $\Delta = \Delta_T = \Delta_{T'}$. Let $q: S^2 \rightarrow C$ be the quotient map that pinches each curve in $\Delta$ to a point, yielding a tree of spheres $C$. Suppose $q(T)$ and $q(T')$ are isotopic in $C$ rel. $q(B) \cup q(\Delta)$. Then there exists a product $t$ of Dehn twists about the curves of $\Delta$ so that $T$ and $t(T')$ are isotopic. 
\end{lemma}

\begin{proof} 
By Remark \ref{remark:ambient isotopy}, there exists an ambient isotopy $I_t$ of $C$ that takes $q(T)$ to $q(T')$, fixing the points $q(B)$ and $q(\Delta)$. Take a small open neighbourhood $U \subset C$ of the points $q(\Delta)$, so that on each component sphere of $C$, $U$ looks like a union of disks that each intersects $q(T)$ in a radius. We can assume without loss of generality that $I_t$ fixes these disks (setwise, not pointwise) throughout.

The preimage $V = q^{-1}(U)$ is a union of tubular neighbourhoods of the curves of $\Delta$. The ambient isotopy $I_t$ on $S^2 \smallsetminus U$ lifts to an ambient isotopy $\widetilde{I}_t$ on $S^2 \smallsetminus V$ that takes $T \smallsetminus V$ to $T' \smallsetminus V$. We can `complete' the isotopy $\widetilde{I}_t$ to all of $S^2$ by interpolating on the tubular neighbourhoods.

The resulting ambient isotopy on $S^2$ takes $T$ to a tree $\widetilde{T}$. The tree $\widetilde{T}$ is equal to $T'$ on $S^2 \smallsetminus V$, but is possibly different on $V$. However, $\widetilde{T}$ and $T'$ intersect a given tubular neighbourhood of $V$ in simple arcs with the same endpoints on $\partial V$. There is thus a power of a Dehn twist in that tubular neighbourhood (about a curve in $\Delta$) taking one arc to the other. Letting $t$ be the product of all these powers of Dehn twists, we see that $T$ is isotopic to the tree $\widetilde{T} = t(T')$. 
\end{proof}

The second technical lemma states that if we act on an embedding $T$ by a product of powers of Dehn twists about curves in $\Delta_T$, then the resulting tree differs from $T$ by a sequence of braid moves.

\begin{lemma}\label{lemma:Dehn twists vs braid moves}
Let $T$ be an embedding of a $B$-marked stable tree $\mathscr{T}$ in $(S^2,B)$, and let $t$ be a Dehn twist about a curve of $\Delta_T$. Then both $t(T)$ and $t^{-1}(T)$ can be obtained from $T$ by a sequence of braid moves.
\end{lemma}

\begin{proof}
Let $t$ be a Dehn twist about a curve $\delta_e$ in $\Delta_T$ and let $h$ be one of the half-edges of $e$. Suppose the vertex $\op{base}(h)$ has valence $k$. Applying a sequence of $(k-1)$ clockwise braid moves at the half-edge $h$ to $T$ yields an embedding $T'$. The tree $T'$ is isotopic to $t(T)$. 

Similarly, applying a sequence of $(k-1)$ anticlockwise braid moves at the half-edge $h$ to $T$ yields an embedding $T'$ isotopic to $t^{-1}(T)$. 
\end{proof}

\begin{subsection}{Half-twists and embedded trees}
Before we present the proof of Proposition \ref{prop:dual trees connected by simple moves}, it will be useful to note some facts about \textit{half-twists}. If $\alpha$ is an arc in $(S^2,B)$, then let $H_\alpha:S^2 \rightarrow S^2$ be the half-twist about $\alpha$ \cite[Section 9.1.3]{farbmargalit2012}.

Suppose $T$ is an embedding with no edges. Let $\alpha_b$ be the simple arc (unique up to isotopy in $S^2 \smallsetminus T$) from $b$ to $\mathit{ord}_T(b)$. Consider the set $$A(T) = \{ H_{\alpha_b} :S^2 \rightarrow S^2 \mid b \in B\}$$ of $|B|$ homeomorphisms. This set generates the impure\footnote{$\op{Mod}_{0,B}$ is the group of self-homeomorphisms of the sphere that fix $B$ setwise (rather than pointwise, as required in the definition of $\op{PMod}_{0,B}$) up to isotopy rel. $B$.} mapping class group $\op{Mod}_{0,B}$ \cite[Section 5.1.3]{farbmargalit2012}. One such half-twist $H_{\alpha_b}$ (and its effect on $T$) is shown in Figure \ref{fig:generators A(T)}. The tree $H_{\alpha_b}(T)$ differs from $T$ by a braid move at the leg $b$.

\begin{figure}[ht!]
\centering
\includegraphics[scale=1]{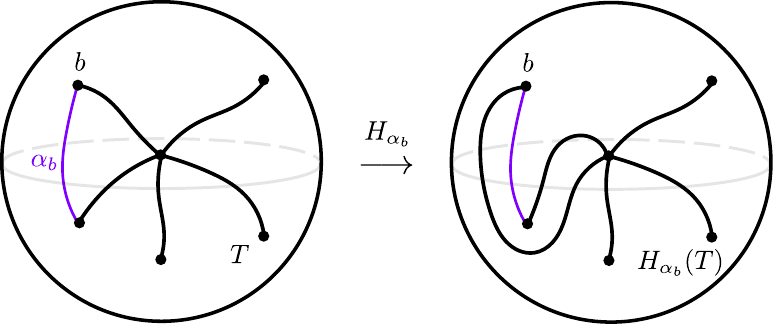}
\caption{The half-twist $H_{\alpha_b}: S^2 \rightarrow S^2$. The tree $H_{\alpha_b}(T)$ differs from $T$ by a braid move at the leg $b$.}
\label{fig:generators A(T)}
\end{figure}

Observe that if $\varphi: S^2 \rightarrow S^2$ is a homeomorphism fixing $B$ setwise, then $H_{\varphi_\alpha} = \varphi \circ H_\alpha \circ \varphi^{-1}$, and so $A(\varphi(T)) = \varphi \circ A(T) \circ \varphi^{-1}$. 
\end{subsection}

\begin{proof}[Proof of Proposition \ref{prop:dual trees connected by simple moves}] 
\begin{sloppypar}
We first deal with the case of trees with no edges. Then we use this case to prove the general case, with the help of Lemma \ref{lemma:lifting a pinched tree} and Lemma \ref{lemma:Dehn twists vs braid moves}.

\textbf{Case 1.} Suppose $\mathscr{T}$ has no edges, so $\emptyset = \Delta_T = \Delta_{T'}$. We wish to find a sequence of half-twists $H_i$ so that 
$$H_N \circ \ldots H_2 \circ H_1 (T) = T'$$
and $H_i \in A(H_{i-1} \circ \ldots \circ H_1(T))$. Then setting $T_i = H_i \circ H_{i-1} \circ \ldots \circ H_1(T)$ gives the desired sequence of trees.

Let $g: S^2 \rightarrow S^2$ be a homeomorphism that takes $T$ to $T'$ (fixing $B$ setwise but not necessarily pointwise). The homeomorphism $g$ is isotopic in $(S^2,B)$ to a product $\overline{H}_N \circ \ldots \circ \overline{H}_1$ of half-twists $\overline{H}_i \in A(T)$, since $A(T)$ generates $\op{Mod}_{0,B}$.

We set
$$H_i = \left(\overline{H}_N \circ \overline{H}_{N-1} \circ \ldots \circ \overline{H}_{N-i+2}\right) \circ \overline{H}_{N-i+1} \circ \left(\overline{H}_{N-i+2}^{-1} \circ \ldots \circ \overline{H}_{N-1}^{-1} \circ \overline{H}_N^{-1}\right).$$
Observe that $H_i \in {A(\overline{H}_N \circ \overline{H}_{N-1} \circ \ldots \circ \overline{H}_{N-i+2}(T))}$. Straightforward cancellations give
$$H_{i-1} \circ \ldots \circ H_1 = \overline{H}_N \circ \overline{H}_{N-1} \circ \ldots \circ \overline{H}_{N-i+2}.$$
Thus, $H_i \in A(H_{i-1} \circ \ldots \circ H_1(T))$ and $H_N \circ \ldots H_2 \circ H_1 (T) = \overline{H}_N \circ \ldots \circ \overline{H}_1 (T) = T'$, as required.

\textbf{Case 2.} Suppose now that $\mathscr{T}$ has at least one edge. By Remark \ref{remark:ambient isotopy}, there is an ambient isotopy of $S^2$ taking $\Delta_{T'}$ to $\Delta_T$, so we may assume that $\Delta = \Delta_T = \Delta_{T'}$.

As in the statement of Lemma \ref{lemma:lifting a pinched tree}, let $q: S^2 \rightarrow C$ be the quotient map that pinches each curve of $\Delta$ to a point, giving a tree of spheres. In each component sphere of $C$, $q(T)$ and $q(T')$ are embeddings with no edges, and so by Case 1 there exists a sequence of braid moves in each component that takes $q(T)$ to $q(T')$.

Applying the same sequence of braid moves to $T$ yields an embedding $T''$. The two embeddings $T'$ and $T''$ satisfy the conditions of Lemma \ref{lemma:lifting a pinched tree}: $T'$ and $T''$ have a common dual multicurve $\Delta$, and $q(T')$ and $q(T'')$ are isotopic in $C$ rel. $q(B) \cup q(\Delta)$. The lemma tells us that there exists a product $t$ of Dehn twists about the curves of $\Delta$ so that $t(T'')$ is isotopic to $T'$. By Lemma \ref{lemma:Dehn twists vs braid moves}, there exists a sequence of braid moves taking $T''$ to $t(T'')$. 

In all, we have found a sequence of braid moves that takes $T = T_0$ to $t(T'') = T_N$, which is isotopic to $T'$, as required. 
\end{sloppypar}
\end{proof}
\end{section}

\newpage 
\bibliographystyle{amsalpha}
\bibliography{{../../references}}

\end{document}